\documentclass[letterpaper,10pt]{IEEEtran}
\usepackage{cite}
\usepackage{amsmath,amssymb,amsfonts}
\renewcommand{\QED}{\(\square\)}
\usepackage{graphicx}
\usepackage{hyperref}
\usepackage{textcomp}
\usepackage[utf8]{inputenc} 
\usepackage[T1]{fontenc}    
\usepackage{url}            
\usepackage{booktabs}       
\usepackage{nicefrac}       
\usepackage{microtype}      
\usepackage{graphicx}
\usepackage{capt-of}
\usepackage{caption}
\usepackage{subcaption}
\usepackage{empheq}
\usepackage{tikz}
\usetikzlibrary{intersections, calc}
\usepackage{xcolor,colortbl}
\def\BibTeX{{\rm B\kern-.05em{\sc i\kern-.025em b}\kern-.08em
    T\kern-.1667em\lower.7ex\hbox{E}\kern-.125emX}}

\newcommand{\setR}{\mathbb{R}}
\newcommand{\setN}{\mathbb{N}}
\newcommand{\setRp}{\setR_{\geq 0}}

\newcommand{\mc}[1]{\mathcal{#1}}

\newcommand{\ones}{\boldsymbol{1}}

\newcommand{\abs}[1]{\left| #1 \right|}
\newcommand{\norm}[1]{\lVert #1 \rVert}

\newcommand{\proj}[1]{\mathbb{P}_{#1}}
\newcommand{\indic}[1]{\mc{I}_{#1}}

\newcommand{\prob}{\mathbb{P}}

\newcommand{\expect}{\mathbb{E}}
\newcommand{\kron}{\otimes}

\DeclareMathOperator*{\argmin}{argmin}
\DeclareMathOperator*{\argmax}{argmax}

\DeclareMathOperator{\ncone}{N}

\DeclareMathOperator{\prox}{prox}
\DeclareMathOperator{\dom}{dom}
\DeclareMathOperator{\ran}{ran}
\DeclareMathOperator{\id}{Id}
\DeclareMathOperator{\zer}{zer}
\DeclareMathOperator{\rint}{ri}

\newtheorem{proposition}{Proposition}

\newtheorem{assumption}{Assumption}

\newtheorem{stassumption}{Standing Assumption}

\newtheorem{lemma}{Lemma}

\newcounter{algorithm}
\newenvironment{algorithm}[1][]
{	\refstepcounter{algorithm}
	\begin{minipage}{\linewidth}
		\medskip
		\hrule
		\smallskip
		\textsc{Algorithm \thealgorithm}. #1
		\smallskip
		\hrule 
		\smallskip
	\end{minipage}
}
{
	\smallskip
	\hrule width\linewidth\relax
	\smallskip
}

\newcommand{\lcm}{E}
\newcommand{\co}{\lcm_0}
\newcommand{\bo}{b_0}
\newcommand{\bp}{b_{\textrm{p}}}
\newcommand{\gs}{w}

\begin{document}
\title{Operator Splitting for Convex Constrained Markov Decision Processes}
\author{Panagiotis D.\ Grontas, Anastasios Tsiamis, John Lygeros
\thanks{
This work was supported as a part of NCCR Automation, a National Centre of Competence in Research, funded by the Swiss National Science Foundation (grant number 51NF40\_225155)}
\thanks{Authors are with the Automatic Control Laboratory, Department of Electrical Engineering and Information Technology,
        ETH Z\"urich, Physikstrasse 3 8092 Z\"urich, Switzerland. (e-mail:{\tt \{pgrontas, atsiamis, jlygeros\}@ethz.ch}). }
}

\maketitle

\begin{abstract}
  We consider finite Markov decision processes (MDPs) with convex constraints and known dynamics.
  In principle, this problem is amenable to off-the-shelf convex optimization solvers, but typically
  this approach suffers from poor scalability.
  In this work, we develop a first-order  algorithm, based on the Douglas-Rachford splitting,
  that allows us to decompose the dynamics and constraints.
  Thanks to this decoupling, we can incorporate a wide variety of convex constraints.
  Our scheme consists of simple and easy-to-implement updates that
  alternate between solving a regularized MDP and a projection.
  The inherent presence of regularized updates ensures last-iterate convergence, numerical stability,
  and, contrary to existing approaches, does not require us to regularize the problem explicitly.
  If the constraints are not attainable, we exploit salient properties of the Douglas-Rachord algorithm
  to detect infeasibility and compute a policy that minimally violates the constraints.
  We demonstrate the performance of our algorithm on two benchmark problems and show
  that it compares favorably to competing approaches.
\end{abstract}


\section{Introduction}
\label{sec:introduction}
\IEEEPARstart{M}{arkov} Decision Processes (MDPs) model an agent that seeks to act optimally in its environment in terms of a scalar cost signal.
This formulation of stochastic optimal control has enjoyed considerable success, all the way from games \cite{silver2017mastering} to fusion reactor
control \cite{degrave2022magnetic}, and is also
the mathematical formalism behind Reinforcement Learning (RL) \cite{sutton2018reinforcement}.
Yet, real-world problems often involve multiple conflicting specifications which can be challenging to incorporate
in a single cost function, e.g., as a weighted sum of task-specific costs.
Such problems can naturally be cast in the framework of constrained MDPs (CMDPs) \cite{altman2021constrained},
by prescribing auxiliary specifications as constraints.
Applications of CMDPs include finance \cite{tamar2012policy, chow2018risk}, power grids \cite{koutsopoulos2011control}
and robotic locomotion \cite{achiam2017constrained}, among others.

Arguably, the most studied constraints are the ones that impose an upper bound on the value function of auxiliary costs,
and correspond to linear constraints in the occupancy measure space.
When the CMDP has finite states and actions, and its dynamics and costs are perfectly known, an optimal policy can be computed by solving a linear program (LP)
\cite{altman2021constrained}.
LP-based approaches find limited applicability as they involve a large number of constraints, which render them computationally challenging
even for small state-action spaces \cite[Sec.~4.7]{sutton2018reinforcement}.
In this vein, \cite{chen2004dynamic, piunovskiy2006dynamic} cast CMDPs as unconstrained MDPs by augmenting the state with continuous variables that represent cumulative auxiliary costs.
The resulting MDP can in principle be solved using dynamic programming (DP) techniques,
but the presence of continuous states renders tabular methods effectively intractable.
Differently, the authors in \cite{khairy2020gradientaware} efficiently solve the dual formulation of CMDPs by exploiting geometric insights,
but their algorithm is tailored to a single linear constraint.

In the RL setting where the CMDP is unknown, a major line of research focuses on primal-dual methods to handle linear constraints
\cite{ding2020natural,achiam2017constrained, liu2021learning}.
The key observation is that, minimizing the Lagrangian for a fixed value of the multipliers associated with the linear constraints
is equivalent to solving an MDP with modified cost \cite{calvofullana2023state}.
Then, optimal multipliers can be computed via dual methods.
Nonetheless reconstructing a primal solution, i.e., an occupancy measure or policy, from a dual solution is generally no
easier than solving the original problem \cite{szepesvari2023constrained} since the Lagrangian is not strictly convex.
Some ways to circumvent this include performing primal averaging \cite{anstreicher2009two}, and
augmenting the MDP state with the multipliers and deploying dual ascent during system operation \cite{calvofullana2023state}.
Another approach is to add strongly convex regularization to the Lagrangian.
For instance, in \cite{gladin2023algorithm} the authors consider a discounted entropy regularization on the policy, 
while in \cite{li2021faster} an additional quadratic regularization of the dual variables is considered.
Since this method solves a modified problem, it is not trivial to ensure that the modified solution is close to the original one.
Further, many safe RL algorithms only provide average-iterate convergence guarantees, i.e., one needs to average
the primal solution over the entire algorithm trajectory.
This is an undesirable property, especially since primal solutions are typically represented via non-linear functions, e.g., neural networks,
thus precluding a straightforward averaging of parameters.
In \cite{muller2023cancellation} an augmented Lagrangian method is employed to achieve last-iterate convergence.
A regularized primal-dual approach based on policy gradient methods is developed in \cite{ding2024lastiterate}, 
along with non-asymptotic last-iterate convergence guarantees.
A host of other approaches exist that are inspired from optimistic mirror descent \cite{moskovitz2023reload}, interior-point methods \cite{liu2020ipo}, or
projection methods \cite{yang2020projection}.

All aforementioned methodologies are tailored to linearly-constrained MDPs.
Yet, MDP formulations with a generic convex objective and constraints have recently gained traction in the RL community \cite{zahavy2021reward}.
For instance, such formulations pertain to apprenticeship learning \cite{abbeel2004apprenticeship},
exploration \cite{hazan2019provably} and diverse skill recovery \cite{eysenbach2018diversity}.
From an algorithmic point of view, the setting of convex constraints is significantly more challenging than linear ones due to the
lack of a straightforward Lagrangian-based reformulation.
In \cite{ying2023policy}, a primal-dual gradient method for convex-constrained MDPs is developed, that exploits the
variational policy gradient theorem \cite{zhang2020variational}.
The authors in \cite{miryoosefi2019reinforcement} use game-theoretic tools to compute feasible policies for convex CMDPs.
The proposed algorithm can also be used to solve convex CMDPs via multiple calls to a feasibility problem \cite{miryoosefi2022simple}.

In this work, we consider finite and known MDPs with convex constraints.
Although this problem setting is amenable to convex optimization solvers, it suffers from the same limitations 
as LP approaches for linearly-constrained MDPs.
To address this problem, we propose a first-order method derived from the 
Douglas-Rachford Algorithm (DRA) \cite[Sec.~28.3]{bauschke_convex_2017} which is equivalent to
the alternating direction method of multipliers \cite{boyd2011distributed, moursi2019note}.
Our contributions are:
\begin{itemize}
    \item To apply the DRA, we introduce a convenient decomposition of the MDP dynamics and the constraints.
    This allows us to exploit potential structure in the constraints and reduce the computational footprint.
    Further, we can deploy our scheme to a wide range of convex constraints, namely any set for which an approximate
    projection can be computed.
    \item Our algorithm involves regularized updates that enjoy numerically favorable properties, without having
    to regularize the original CMDP or its Lagrangian.
    This inherent regularization facilitates recovering a primal solution and, thus,
    ensures last-iterate convergence.
    \item We study problem settings where the specified convex constraints are not attainable, and
    we exploit salient properties of the DRA to endow our approach with an infeasibility detection mechanism
    and meaningful characterization of its iterates.
    To the best of our knowledge, there are limited works that address infeasible CMDPs, see e.g., 
    \cite{miryoosefi2019reinforcement,       miryoosefi2022simple, yu2021provably}.
\end{itemize}
Finally, we deploy our algorithm on two well-studied benchmark problems. 
We verify its ability to rapidly retrieve medium-accuracy solutions, outperform popular competing methodologies,
and we motivate the usefulness of our infeasibility detection analysis with intuitive numerical examples.

\subsubsection*{Notation}
We denote by \( \setN, \setR,\text{ and } \setRp \) the set of natural, real, and non-negative real numbers, respectively.
We use \( \norm{\cdot} \) to denote the Euclidean norm.
We let \( \ones_n \in \setR^n \) be the vector of all ones.
For a finite set \(\mc{S}\) with cardinality \( \abs{\mc{S}} \), 
we denote the unit simplex by \( \Delta(\mc{S}) \coloneqq \{ x \in \setRp^{\abs{\mc{S}}} \,|\, \ones_{\abs{\mc{S}}}^{\top} x = 1 \} \).
The indicator function of a set \(\mc{X}\subseteq \setR^n\) is \(\indic{\mc{X}}(x) = 0 \) if \( x \in \mc{X} \),
and equals \( + \infty \) otherwise.
A function \( f : \setR^n \to \setR \cup \{ -\infty, +\infty \}\) is called proper
if \( \dom f \coloneqq \{ x \in \setR^n \,|\, f(x) < +\infty \} \neq \emptyset \),
and \( f \) does not take the value \( - \infty \).
The proximal operator of a proper, closed, convex function \(f\) with parameter
\(\sigma > 0 \) is given by \( \prox_{\sigma f}(x) \coloneqq \argmin_{y} \{ f(y) + \frac{1}{2 \sigma} \norm{x - y}^2 \} \).

\section{Preliminaries and Problem Formulation}
\subsection{Markov Decision Processes}
We consider infinite-horizon discounted tabular MDPs, specified by the tuple
\( (\mc{S}, \mc{A}, P, c, \gamma, \rho) \), 
where \( \mc{S} \) and \( \mc{A} \) are the finite sets of states and actions with respective
cardinality \(S \coloneqq \abs{\mc{S}} \) and \( A \coloneqq \abs{\mc{A}} \),
\( P : \mc{S} \times \mc{A} \to \Delta(\mc{S}) \) is the transition kernel,
\( c : \mc{S} \times \mc{A} \to \setR \) is the cost function,
\( \gamma \in (0, 1) \) is the discount factor, and
\( \rho \in \Delta(\mc{S}) \) is the initial state distribution.
We represent \( P \in \setR^{S A \times S} \) as a row-stochastic matrix,
and \( c \in \setR^{SA}, \rho \in \Delta(\setR^S) \) as vectors,
and we enumerate state-action pairs lexicographically as \( (s_1, a_1), (s_2, a_1), \) and so on.
At each time step \( t \), the agent is at state \( s_t \) and chooses action \( a_t \).
As a result, it transitions to a next state \( s_{t+1} \) with probability \( P(s_{t+1} \, | \, s_t, a_t) \)
and incurs a cost \( c(s_t, a_t) \).
The agent's goal is to minimize the cumulative discounted cost, assuming that the initial state is distributed according to \( \rho \).

A stationary Markov policy is a mapping \( \pi : \mc{S} \to \Delta(\mc{A}) \) such that
the probability of applying action \( a \) in state  \(s \) is \( \pi(a \,|\, s) \).
The value function of a policy \( \pi \) associated with the cost \( c \) starting from the initial state \( s \)
is \( V_{\pi}^c(s) \coloneqq (1-\gamma) \expect^{\pi} \big[ \sum_{t=0}^{\infty} \gamma^t c(s_t, a_t) \,|\, s_0 = s \big]  \), 
where the expectation is taken with respect to system trajectories under \( \pi \).
For conciseness, we let \( V_{\pi}^c \coloneqq \expect_{s \sim \rho}[ V_{\pi}^c(s)]  \).
Solving an MDP amounts to finding a policy \( \pi^{\star} \) such that
\begin{equation} \label{eq:MDP_with_policies}
    \pi^{\star} \in \argmin_{\pi \in \Pi} V_{\pi}^c,
\end{equation}
where \( \Pi \) denotes the set of all stationary Markov policies.

The non-convex optimal control problem \eqref{eq:MDP_with_policies} can equivalently be cast
as a convex one by introducing the (normalized discounted) occupancy measure of \( \pi \),
which is defined as \( d_{\pi}(s, a) \coloneqq (1 - \gamma) \sum_{t = 0}^{\infty} \gamma^t
\prob^{\pi}[s_t = s, a_t = a \,|\, s_0 \sim \rho ]\).
Intuitively, \( d_{\pi}(s, a) \) measures the discounted frequency of observing the pair \( (s,a) \).
Any \( d_{\pi} \) takes values in the bounded polytope
\begin{equation} \label{eq:occupancy_measure_set}
    \begin{aligned}
        \mc{D} \coloneqq \Big\{ d \in \setRp^{SA} ~|~ & \sum_{a} d(s, a) =
        (1 - \gamma) \rho(s) \\ &+ \gamma \sum_{a', s'} P(s \,|\, s', a') d(s', a')  
        \Big\} ~
        ;        
    \end{aligned}
\end{equation}
indeed, \( \mc{D} = \{ d_{\pi} \,|\, \pi \in \Pi \} \).
Conversely, given any \( d \in \mc{D} \), the policy \( \pi(a \,|\, s) = d(s,a) / \sum_{a} d(s, a) \) induces an occupancy
measure \(d_{\pi} \) such that \( d_{\pi} = d \) \cite[Th.~3.2]{altman2021constrained};
if \( \sum_{a} d(s, a) = 0\), then \( \pi(\cdot \,|\, s) \) can be arbitrary.
By definition of \( d_{\pi} \), \( V_{\pi}^c = c^{\top} d_{\pi} \),
which allows us to reformulate \eqref{eq:MDP_with_policies} as the following LP:
	\begin{subequations} \label{eq:occupancy_mdp}
		\begin{alignat}{2}
		&\underset{ d \in \setR^{SA}}{\mathclap{\mathrm{minimize}}} 
		\quad~ && \quad c^{\top} d \\
		& \overset{\hphantom{ d \in \setR^{SA}}}{\mathclap{\mathrm{subject~to}}} \quad~
		&& \quad d \in \mc{D}~ .
		\end{alignat}
	\end{subequations}
 
\subsection{Convex-Constrained Markov Decision Processes}
We consider convex-constrained MDPs specified as the following convex program:
\begin{subequations} \label{eq:occupancy_cmdp}
    \begin{alignat}{2}
    &\underset{ d \in \setR^{SA}}{\mathclap{\mathrm{minimize}}} 
    \quad~ && \quad c^{\top} d \\
    & \overset{\hphantom{d \in \setR^{SA}}}{\mathclap{\mathrm{subject~to}}} \quad~
    && \quad d \in \mc{C}\cap\mc{D},
    \end{alignat}
\end{subequations}
where the set \( \mc{C} \) satisfies the following standing assumption, 
which is valid throughout the paper.
\begin{stassumption} \label{stassumption:setc}
    The set \( \mc{C} \subseteq \setR^{SA} \) is non-empty, closed, and convex.
\end{stassumption}

For concreteness, we adopt a representation of the form
\( \mc{C} = \{ d \in \setR^{SA} \, | \, C_i(d) \leq b_i, ~ i = 1, \ldots, n_c \}  \),
where each \( C_i \) is a convex real-valued function and \( b_i \) is a scalar.
Depending on the structure of \( \mc{C} \), we can prescribe various
performance or safety specifications, as discussed next.

\textbf{Example 1:}
Arguably, the most common setting corresponds to polytopic sets of the form
\( \mc{C} = \{ d \in \setR^{SA} \,|\, \lcm d \leq b \} \),
where \( \lcm \in \setR^{n_c \times SA}, b \in \setR^{n_c} \).
We let \( \lcm_i \in \setR^{SA} \) denote the rows of \( \lcm \), and 
observe that \( \expect_{s \sim \rho}[V^{\lcm_i}_\pi(s)] =  \lcm_i^{\top} d_\pi \leq b_i \).
Therefore, we can interpret \( \lcm_i \) as an auxiliary cost, whose associated value function
is upper-bounded by \( b_i \).

\indent\textbf{Example 2:}
We can encode imitation learning objectives via \( \ell_p \)-ball constraints such as \( \mc{C} = \{ d \in \setR^{SA} \,|\, \norm{d - \hat{d}}_p \leq \varepsilon \} \),
where \( \varepsilon > 0 \), \( p \in [1, +\infty] \), and \(\hat{d} \in \mc{D} \).
Intuitively, we require the optimal occupancy measure to remain close to some expert measure \( \hat{d} \). 
Such constrains could also arise in a setting where we have a nominal MDP model that is only reliable in
the neighborhood of a nominal policy.

\textbf{Example 3:}
Using the convex set \( \mc{C} = \{ d \in \setR^{SA} \,|\, \sum_{s,a} d(s,a) \log(d(s,a)) \geq \varepsilon \} \) 
involving the entropy of the occupancy measure
we can promote exploratory behavior, e.g., when using a nominal MDP model to obtain new data. 

\textbf{Example 4:} Non-linear convex constraints can be used to ensure robustness of the induced policy.
For example, assume that \( c \) is only an approximation of the true cost function, which is unknown but guaranteed to live on the convex set \( \mc{A} \).
Then, we can control the worst-case performance degradation due to the uncertain cost by enforcing
\( d \in \mc{C} = \{ d \in \setR^{SA} \,|\, \sup_{c' \in \mc{A}} (c' - c)^{\top} d \leq \varepsilon \} \).

For various types of constraints \( \mc{C} \), one could solve \eqref{eq:occupancy_cmdp} as a convex program using existing solvers,
e.g., Example 1 and Example 2, with \( p = 1\) or \( p = +\infty\), give rise to an LP,
while Example 2 with \( p = 2\) corresponds to a second-order cone program.
But, as noted in the introduction, these approaches are significantly less scalable than DP-based methods.
The latter are solely applicable to linearly-constrained MDPs and come with their own set of challenges,
namely state augmentation with continuous variables or average-iterate convergence.
To bridge this gap, we aim to derive a DP-inspired algorithm that works for generic convex constraints.

\section{Algorithm Design} \label{sec:algorithm_design}
Our main algorithmic idea is to decouple the dynamics of the problem in \( \mc{D} \) from the
constraints in \( \mc{C} \), and then exploit the individual structure of the two parts to derive an efficient and modular algorithm.
To achieve this, we draw inspiration from distributed optimization techniques and, specifically, we deploy the
Douglas-Rachford Algorithm (DRA).
The DRA is used to solve composite optimization problems of the form
\begin{equation} \label{eq:composite_convex_optimization}
    \underset{\displaystyle x}{\mathclap{\mathrm{minimize}}} 
    \quad~ \quad f(x) + g(x),
\end{equation}
where \( f \) and \( g \) are convex, closed, and proper functions.
The iterations of the DRA require evaluating the proximal operators of \( f \) and \( g \).
Therefore, splitting the sum \( f + g \) onto individual functions \( f \) and \( g \) is a critical design choice for the efficiency of the algorithm.

We express the CMDP in \eqref{eq:occupancy_cmdp} in composite form as follows:
\begin{equation} \label{eq:composite_form_cmdp}
        \underset{\displaystyle d \in \setR^{SA}}{\mathclap{\mathrm{minimize}}} 
        \quad~ \underbrace{c^{\top} d + \indic{\mc{D}}(d)}_{\eqqcolon f(d)} + \underbrace{\indic{\mc{C}}(d)}_{\eqqcolon g(d)}~.
\end{equation}
Notice that our choice of \( f \) corresponds to the objective of an unconstrained MDP as in \eqref{eq:occupancy_mdp}.
A similar idea is explored in \cite{o2013splitting} in the context of optimal control of linear deterministic systems.
Applying the DRA for the proposed splitting yields the following iteration, initialized with \( \gs_0 \in \setR^{SA}\):
\begin{subequations} \label{eq:dra_for_cmdps}
    \begin{empheq}[left = (\forall k \in \setN) \quad \empheqlfloor]{align}
    d_{k} & = \prox_{\sigma f} (\gs_k) 
    \label{eq:dra_regularized_mdp} \\
    \nu_k & = \frac{1}{\sigma} (\gs_k - d_k)
    \label{eq:dra_dual_update}	\\
    {z}_{k} & = \prox_{\sigma g}(2 d_k - \gs_k) 
    \label{eq:dra_safety_projection} \\
    \gs_{k+1} & = \gs_k + (z_k - d_k)
    \label{eq:dra_governing_update}
    \end{empheq}
\end{subequations}
where \( \sigma > 0\) is a scaling parameter.
Intuitively, the algorithm proceeds by performing proximal minimization of \( f \) and \( g \) in an alternating fashion, and
uses \( \gs_k \) to integrate the difference of the minimizers.
In \autoref{sec:convergence_analysis} we will show that, under appropriate conditions, the pair of iterates \( (d_k, \nu_k) \)
converges to a primal-dual solution of \eqref{eq:occupancy_cmdp}.

Implementing \eqref{eq:dra_for_cmdps} is not straightforward as it requires solving the subproblems
\eqref{eq:dra_regularized_mdp} and \eqref{eq:dra_safety_projection}.
Next, we elaborate on how to perform these updates efficiently.
All derivations for the following subsections are given in Appendix \ref{app:derivation_of_qrpi}.

\subsection{Quadratically-regularized MDPs}
The update \eqref{eq:dra_regularized_mdp} corresponds to solving an MDP with quadratic regularization:
\begin{subequations} \label{eq:regularized_mdp}
    \begin{alignat}{2}
        &\underset{\displaystyle d \in \setR^{SA}}{\mathclap{\mathrm{minimize}}} 
        \quad~ && c^{\top} d + \frac{1}{2 \sigma}\norm{d - \gs_k}^2 \\
        & \overset{\hphantom{\displaystyle d \in \setR^{SA}}}{\mathclap{\mathrm{subject~to}}} \quad~
        && \Xi^{\top} d = (1 -\gamma) \rho + \gamma P^{\top} d,
        \label{eq:regularized_mdp_flow_constraint} \\
        & && d(s,a) \geq 0, ~ \forall s \in \mc{S}, a \in \mc{A}, \label{eq:regularized_mdp_positivity}
    \end{alignat}
\end{subequations}
where, for conciseness, we define the matrix \( \Xi^{\top} \coloneqq \ones_A^{\top} \kron I_S \in \setR^{S \times SA} \) that acts on \( d \)
as \( [\Xi^{\top} d]_s = \sum_{a} d(s, a) \).
Directly solving the quadratic program \eqref{eq:regularized_mdp} is disadvantegeous from a computational perspective
due to the large number of affine constraints.
Instead, we pursue an approach inspired from the regularized MDP literature \cite{neu2017unified, geist2019atheory} that exploits the dynamical system
interpretation of the constraints \eqref{eq:regularized_mdp_flow_constraint}, \eqref{eq:regularized_mdp_positivity}.

To do so, we introduce the Largrange multipliers \( V \in \setR^S \) and \( \varphi \in \setR^{SA} \)
associated with \eqref{eq:regularized_mdp_flow_constraint} and \eqref{eq:regularized_mdp_positivity}, respectively.
Then, we derive (see Appendix \ref{app:derivation_of_qrpi}) the dual of \eqref{eq:regularized_mdp} that reads:
\begin{equation} \label{eq:dual_regularized_mdp}
	\begin{aligned}
		&\underset{ V \in \setR^{S},\, \varphi \in \setRp^{SA}}{\mathclap{\mathrm{maximize}}} 
		\quad \kappa(V, \varphi):= - \frac{\sigma}{2} \norm{c + \gamma P V - \Xi V - \varphi}_2^2 
            \\
		& ~~~~~~
            + \gs_k^{\top}(c + \gamma P V - \Xi V - \varphi) + (1 - \gamma) \rho^{\top} V.
    \end{aligned}
\end{equation}
Recall that, in the primal and dual LP formulation of MDPs without regularization \cite[Subsec.~3.2]{altman2021constrained}, 
the multipliers of the equality constraints \eqref{eq:regularized_mdp_flow_constraint} correspond to the value function.
For this reason, in the context of \eqref{eq:regularized_mdp} we interpret \( V \) as a regularized value function,
similarly to \cite{neu2017unified}.
Further, complementary slackness dictates that \( 0 \leq \varphi \perp d \geq 0 \), hence, \( \varphi(s,a) > 0\) implies that \( d(s,a) = 0\)
for any state-action pair \( (s,a) \).
As such, \( \varphi \) acts as a dual occupancy measure that indicates which state-action pairs are never visited.

With these connections in mind, we observe that given an optimal \( \varphi^{\star} \), we can retrieve the corresponding optimal value function
\( V^{\star} \) by maximizing \eqref{eq:dual_regularized_mdp} with respect to \( V \).
Indeed, setting \( \nabla_V \kappa(V, \varphi^{\star}) \) to zero and solving for \(V\) yields a linear system of equations, akin to policy evaluation:
\begin{equation} \label{eq:regularized_value_optimality}
    (\gamma P - \Xi)^{\top} (\gamma P - \Xi) V =
		(\gamma P - \Xi)^{\top} (\gs_k/\sigma - c + \varphi^{\star}) + \frac{1}{\sigma} (1 - \gamma) \rho.
\end{equation}
Noticing that \( \gamma P - \Xi \) has trivial nullspace (see \autoref{lemma:trivial_nullspace} in Appendix \ref{app:derivation_of_qrpi}), we deduce that
\eqref{eq:regularized_value_optimality} has a unique solution.
Conversely, in a policy improvement fashion, given an optimal regularized value function \( V^{\star} \) we derive the corresponding
\( \varphi^{\star} \) by maximizing \eqref{eq:dual_regularized_mdp} as follows:
\begin{equation} \label{eq:dual_occupancy_optimality}
    \varphi^{\star} = \max(c - \Xi V^{\star}  + \gamma P V^{\star} - \gs_k/\sigma, 0).
\end{equation}
Finally, we can retrieve the solution of \eqref{eq:regularized_mdp} from the dual solution by exploiting strong duality:
\begin{equation} \label{eq:regularized_mdp_solution}
   d^{\star} = \sigma \max(- c + \Xi V^{\star} - \gamma P V^{\star} + \gs_k / \sigma, 0).
\end{equation}

The expressions in \eqref{eq:regularized_value_optimality}, \eqref{eq:dual_occupancy_optimality}, and \eqref{eq:regularized_mdp_solution} motivate solving
\eqref{eq:dual_regularized_mdp} through the following iterative scheme, which we refer to as quadratically regularized policy iteration (QRPI):
\begin{subequations} \label{eq:qrpi}
    \begin{align}
    \begin{split}
    V_{\ell+1}^{\textrm{in}} & \leftarrow ((\gamma P - \Xi)^{\top} (\gamma P - \Xi))^{-1} \\
    & ~~~\big((\gamma P - \Xi)^{\top} (\frac{1}{\sigma} \gs_k - 
    c + \varphi_{\ell}^{\textrm{in}}) + \frac{1}{\sigma} (1 - \gamma) \rho \big)
    \label{eq:inner_v_update}
    \end{split} \\
    \varphi_{\ell + 1}^{\textrm{in}} & \leftarrow \max(c + \gamma P V_{\ell + 1}^{\textrm{in}} - \Xi V_{\ell + 1}^{\textrm{in}} - \frac{1}{\sigma} \gs_k, 0)
    \label{eq:inner_varphi_update} \\
    d_{\ell+1}^{\textrm{in}} & \leftarrow \sigma \max(- c - \gamma P V_{\ell + 1}^{\textrm{in}} + \Xi V_{\ell + 1}^{\textrm{in}}  + \frac{1}{\sigma} \gs_k, 0)~.
    \label{eq:inner_d_update}
    \end{align}
\end{subequations}
As a remark, QRPI can be shown to be a quasi-Newton method to solve \eqref{eq:regularized_mdp}.
Similarly, in standard MDPs, policy iteration is known to be a Newton method to solve the Bellman optimality equation \cite{gargiani2022dynamic},
and has inspired quasi-Newton approaches as in \cite{kolarijani2023optimization}.

Notice that, QRPI forms an inner loop which is used to solve subproblem \eqref{eq:dra_regularized_mdp} of the DRA.
We highlight this by using a distinct index, \( \ell \in \setN \), and the superscript "\textrm{in}" 
for the inner loop iterates.
In \autoref{sec:convergence_analysis} we establish convergence of \eqref{eq:qrpi} to a primal-dual solution of \eqref{eq:regularized_mdp}.

Interestingly, notice that \eqref{eq:inner_varphi_update} and \eqref{eq:inner_d_update} can be expressed
as \( \varphi^{\textrm{in}}_{\ell+1} = \max(A(V^{\textrm{in}}_{\ell+1}), 0) \) 
and \( d^{\textrm{in}}_{\ell+1} = \sigma \max(-A(V^{\textrm{in}}_{\ell+1}), 0) \),
where \( A(V) \coloneqq c + \gamma P V - \Xi V - \gs_k / \sigma \).
Explicitly writing the components \( [A(V)]_{(s,a)} = c(s,a) + \gamma \sum_{s'} P(s' \,|\, s,a) V(s') - V(s) - \gs_k(s,a)/\sigma\),
we observe that \( A \) is a (dis)advantage function, encoding the dynamics and cost, plus a term arising from the quadratic regularization,
encoding a preference towards \( \gs_k \).
Further, we can decompose the positive and negative entries of \( A(V) \) onto \( \varphi \) and
\( d \), respectively, as indicated in \autoref{subfigure:decomposition_of_advantage}.

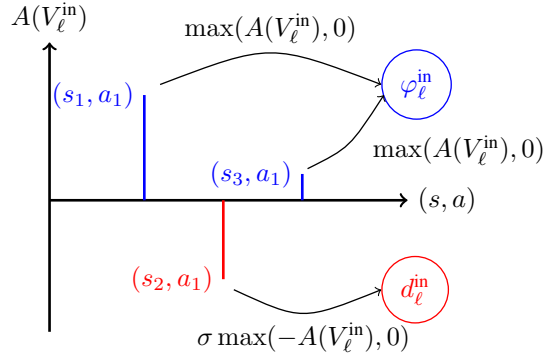
\begin{figure}[htbp]

        \centering
  	\begin{tikzpicture}[scale=0.70]
		\draw[->, line width=1] (-0.3,0) -- (6.5,0) node[right] {$(s, a)$};
		\draw[->, line width=1] (-0.3,-2.5) -- (-0.3,3) node[above] {$A(V_{\ell}^{\textrm{in}})$};
		
		
		\draw[blue, line width=1pt] (1.5,0) -- (1.5,2)
		node[left]{$(s_1, a_1)$};
		
		\draw[black, line width=0.5pt,->] (1.8,2.2) .. controls (3.9,3.0) ..
		node[midway, above] {$\max(A(V_{\ell}^{\textrm{in}}), 0)$} (6.0,2.2)
		node[right, blue, draw, circle, fill=white] {\(\varphi_{\ell}^{\textrm{in}}\)};
		
		\draw[blue, line width=1pt] (4.5,0) -- (4.5,0.5)
		node[left]{$(s_3, a_1)$};
		
		\draw[black, line width=0.5pt,->] (4.6,0.6) .. controls (5.4,1.0) ..
		node[pos=0.75, below right] {$\max(A(V_{\ell}^{\textrm{in}}), 0)$} (6.05,2.0);
		
		\draw[red, line width=1pt] (3.0,0) -- (3.0,-1.5)
		node[left]{$(s_2, a_1)$};
		
		\draw[line width=0.5pt,->] (3.2,-1.7) .. controls (4.5, -2.3) ..
		node[midway, below] {$\sigma \max(-A(V_{\ell}^{\textrm{in}}), 0)$} (6.0,-1.7) node[right, red, draw, circle, fill=white] {\( d_{\ell}^{\textrm{in}}\)};;
	\end{tikzpicture}
        \caption{Decomposition of \(A(V)\) onto \( d \) and \( \varphi \).
        Informally, if \( [A(V_{\ell}^{\text{in}})]_{(s,a)} < 0\) then playing action \( a \) at state \( s \) will improve performance, up to a slack
        of \( \gs_k/\sigma \), therefore, we use it to compute \( d_{\ell}^{\text{in}} \).
        Conversely, state-action pairs satisfying \( [A(V_{\ell}^{\text{in}})]_{(s,a)} > 0\) are undesirable as they would deteriorate performance, hence
        are placed in the dual occupancy \( \varphi_{\ell}^{\text{in}} \).
        }
        \label{subfigure:decomposition_of_advantage}
\end{figure}

\subsection{Constraint Projection}
We express the update \eqref{eq:dra_safety_projection} in a more familiar form, by recalling that \( g = \indic{\mc{C}}\), as follows:
\begin{equation}
    \label{eq:prox_g_is_projection}
    \begin{aligned}
        \prox_{\sigma g}(2 d_k - \gs_k) & = \argmin_{d' \in \mc{C}}~ \norm{d' - (2 d_k - \gs_k)}^2 \\ 
         & = \proj{\mc{C}}(2 d_k - \gs_k)~.
    \end{aligned}
\end{equation}
The projection onto a closed convex set is a well-studied operation and admits efficient, often closed-form, implementation for
various convex sets of interest, such as norm balls \cite[Examples 3.18, 29.27]{bauschke_convex_2017}, convex cones \cite[Examples  6.29, 29.32]{bauschke_convex_2017} or sublevel sets of support functions.

As previously highlighted, polyhedral constraints are of particular interest yet, in general, do not admit a closed-form projection.
Specifically, for polyhedral sets \( \mc{C} = \{ d \in \setR^{SA} \,|\, \lcm d \leq b \} \), 
where \( \lcm \in \setR^{n_c \times SA}\), the projection \eqref{eq:prox_g_is_projection}
amounts to solving a quadratic program.
Yet, the dimensionality of \( d \) renders solving \eqref{eq:prox_g_is_projection} at each iteration expensive.
Nonetheless, for most interesting problems it is often the case that \( n_c \ll S A \). 
Hence, it is advantageous to
solve the, significantly smaller, dual of \eqref{eq:prox_g_is_projection} instead:
\begin{equation}
    \lambda^{\star} \in \underset{\displaystyle \lambda \in \setRp^{n_c}}{\mathclap{\mathrm{argmax}}} 
    ~~  - \frac{1}{4} \lambda^{\top} \lcm \lcm^{\top} \lambda + (\lcm (2 d_k - \gs_k) - b)^{\top} \lambda,
\end{equation}
and retrieve the projection as \( \proj{\mc{C}}(2 d_k - \gs_k) = 2 d_k - \gs_k - \frac{1}{2} \lcm^{\top} \lambda^{\star} \).

\subsection{Overall Algorithm}
Combining the general DRA in \eqref{eq:dra_for_cmdps} with the specific update rules \eqref{eq:qrpi} and \eqref{eq:prox_g_is_projection}
we derive \autoref{alg:os-cmdp}, which we refer to as Operator Splitting for Constrained MDPs (OS-CMDP).
It comprises a double loop: the outer loop iteration, indexed by \( k \), implements the DRA, 
while the inner loop iteration, indexed by \( \ell \), (approximately) solves the regularized
MDP \eqref{eq:regularized_mdp} by performing \( \overline{\ell} \) iterations of QRPI \eqref{eq:qrpi}.
Design guidelines for the hyperparameters \( \overline{\ell}\) and \( \sigma \) are discussed in \autoref{sec:practical_implementation}.
Note that \autoref{alg:os-cmdp} keeps track of \( \varphi_k \), i.e., the optimal dual occupancy for
\eqref{eq:regularized_mdp} given \( \gs_k \). 
Then, we use \( \varphi_k\) to warm-start QRPI to solve \eqref{eq:regularized_mdp} for \(\gs_{k+1} \).
In the next section, we study the asymptotic behavior of \autoref{alg:os-cmdp} 
and establish appropriate termination criteria.

\begin{figure}
\begin{algorithm}[Operator Splitting for Constrained MDPs] \label{alg:os-cmdp}%
\textbf{Inputs:} \( \sigma > 0\), \(\overline{\ell} \in \setN \). \\
\textbf{Initialization:} \( k \leftarrow 0 \), \( \gs_0 \leftarrow 0 \), \( \varphi_{-1} \leftarrow 0 \). \\
\textbf{Repeat until terminal criterion:} \\
\(
\left \lfloor
\begin{array}{l}
    \text{Solve Regularized MDP via QRPI:} \\
    \left|
    \begin{array}{l}
        \textit{Initialization: } \ell \leftarrow 0 \\
        \textit{Warmstart: } \varphi_{0}^{\textrm{in}} = \varphi_{k-1} \\
        \textit{For \( \ell = 0, \ldots, \overline{\ell} - 1\):} \\
        \left \lfloor
        \begin{array}{l}
            V_{\ell+1}^{\textrm{in}}  \leftarrow \text{Update using } \eqref{eq:inner_v_update} \\
            \varphi_{\ell+1}^{\textrm{in}} \leftarrow \text{Update using } \eqref{eq:inner_varphi_update} \\
            d_{\ell+1}^{\textrm{in}} \leftarrow \text{Update using } \eqref{eq:inner_d_update} 
        \end{array}
        \right. \\[2em]
        \text{Update outer loop variables:} \\
        \varphi_{k} \leftarrow \varphi_{\overline{\ell}}^{\textrm{in}} \\
        d_{k} \leftarrow d_{\overline{\ell}}^{\textrm{in}}
    \end{array}
    \right. \\[.5em]
    \text{Dual update:} \\
        \left|
        \begin{array}{l}
            \nu_k \leftarrow \frac{1}{\sigma}(\gs_k - d_k)
        \end{array}
        \right. \\[.5em]
    \text{Project onto } \mc{C}\text{:} \\
        \left|
        \begin{array}{l}
            z_k \leftarrow \proj{\mc{C}}(2 d_k - \gs_k)
        \end{array}
        \right. \\
    \text{Auxiliary update:} \\
    \left|
    \begin{array}{l}
        \gs_{k+1} \leftarrow \gs_k + (z_k  -  d_k)
    \end{array}
    \right. \\[.3em]
    k \leftarrow k + 1
\end{array}
\right. \\[.3em]
\textbf{Output: } d_k
\)
\end{algorithm}
\end{figure}

\section{Convergence Analysis} \label{sec:convergence_analysis}
In this section, we establish convergence of QRPI, in the inner loop of \autoref{alg:os-cmdp}, to a primal-dual solution of \eqref{eq:regularized_mdp}.
Then, we study the asymptotic behavior of \autoref{alg:os-cmdp} for both feasible and infeasible instances of \eqref{eq:occupancy_cmdp},
which we employ to design termination criteria.
The proofs for the technical statements in this section are provided in Appendix \ref{app:convergence_proofs}.

\subsection{Convergence of QRPI}
Our convergence proof hinges on the fact that the dual updates \eqref{eq:inner_v_update}, and \eqref{eq:inner_varphi_update} of QRPI
can be viewed as applying block coordinate maximization to the dual problem \eqref{eq:dual_regularized_mdp}.
In particular, the dual updates are equivalent to:
\begin{subequations}
        \begin{empheq}[left = \forall \ell \in \setN: \empheqlfloor \label{eq:block_coordinate_ascent}]{align}
        V^{\ell+1} & = \argmax_V ~ \kappa(V, \varphi^{\ell}) \label{eq:V_update} \\
        \varphi^{\ell+1} & = \argmax_{\varphi \in \setRp^{SA}} ~ \kappa(V^{\ell+1}, \varphi), \label{eq:phi_update}
        \end{empheq}
\end{subequations}
where \( \kappa \) is the dual function in \eqref{eq:dual_regularized_mdp}.
The primal updates simply follow from the relation between primal and dual solutions established in \eqref{eq:regularized_mdp_solution}.
\begin{proposition} \label{prop:qrpi_convergence}
    For any \( \gs_k \in \setR^{SA} \) and \( \varphi_0 \in \setR^S \)
    the sequence \( (d_{\ell}^{\textrm{in}}, V_{\ell}^{\textrm{in}}, \varphi_{\ell}^{\textrm{in}})_{\ell \in \setN} \) generated by
    \eqref{eq:qrpi} converges to a primal-dual solution of \eqref{eq:regularized_mdp} with R-linear rate.
    In particular, \( d_{\ell}^{\textrm{in}} \to \prox_{\sigma f} (\gs_k) \).
\end{proposition}

\subsection{Feasible Problem Instances}
Now, we turn our attention to the asymptotic behavior of \autoref{alg:os-cmdp} under the following assumption.
\begin{assumption} \label{assum:feasible_constraint_qualification}
    One of the following holds:
    (i) \(\mc{D} \cap \rint \mc{C} \neq \emptyset \), where \( \rint \) denotes the relative interior;
    (ii) \( \mc{C} \) is a polyhedron and \( \mc{D} \cap \mc{C} \neq \emptyset \).
\end{assumption}
Assumptions \ref{assum:feasible_constraint_qualification}\textit{(i)} and \ref{assum:feasible_constraint_qualification}\textit{(ii)} are
standard and, respectively, correspond to strict feasibility and feasibility.

If we allow full convergence of the inner QRPI loop, i.e., we set \( \overline{\ell} = + \infty \), then
we can readily establish convergence of \autoref{alg:os-cmdp} as an instance of the DRA.
\begin{proposition} \label{prop:nominal_exact_convergence}
    Under \autoref{assum:feasible_constraint_qualification},
    the sequence \( (d_k, z_k, \nu_k)_{k \in \setN}  \) generated by \autoref{alg:os-cmdp}, where the inner problem is solved to optimality,
    converges to some \( (d^{\star}, d^{\star}, \nu^{\star}) \), where \( d^{\star} \) is a primal and \( \nu^{\star} \) is a dual 
    solution of \eqref{eq:occupancy_cmdp}.
\end{proposition}

In practice, we observe that using a small finite \( \overline{\ell} \), instead of solving the inner problem to optimality, 
yields similar performance, as we discuss in \autoref{sec:practical_implementation} and Appendix \ref{appendix:inexact_qrpi}.

\subsection{Infeasible Problem Instances}
Next, we focus on cases where the constraint set \( \mathcal{C} \) is not compatible with the dynamics, 
i.e., \( \mc{C} \cap \mc{D} = \emptyset\).
This means that there exists no occupancy measure, or equivalently Markov policy, that satisfies the specifications
in \( \mathcal{C} \).
In this setting, it is well-known that 
the sequence \( ( \gs_k )_{k \in \setN} \) diverges, but,
nonetheless, we can extract useful information from the rate at it which it does so \cite{banjac2021minimal}.
\begin{proposition} \label{prop:infeasibility_governing_sequence}
    Let \( v \coloneqq \argmin_{\beta \in \mc{D} - \mc{C}} \norm{\beta} \). It holds true that
    \( \gs_{k} - \gs_{k+1} \to v \).
\end{proposition}

The vector \( v \), that corresponds to the so-called minimal displacement vector \cite{bauschke2023douglas}, represents the minimum-norm
translation of \( \mc{C} \) such that \( \mc{C} \cap \mc{D} \neq \emptyset \).
Intuitively, a system designer can interpret \( v \) as the direction along which the constraints should be
relaxed to render the problem feasible.
Moreover, it follows that \( v \neq 0 \) if and only if 
\( \mc{C} \cap \mc{D} = \emptyset \).
In fact, it can be shown that if the problem is infeasible the non-zero \( v \)  generates a strongly separating hyperplane
between \( \mc{C} \) and \( \mc{D} \), i.e.,
\( \min_{d \in \mc{D}} v^{\top} d > \sup_{z \in \mc{C}} v^{\top} z \)
(see \autoref{lemma:strongly_separating_hyperplane} in Appendix \ref{app:convergence_proofs}). 
Similar results have been shown in the context of quadratic and conic programs \cite{raghunathan2014infeasibility, liu2017new, banjac2019infeasibility}.

Despite infeasibility, we can still extract a meaningful policy from the limit of \( (d_k )_{k \in \setN} \) by exploiting
the asymptotic properties of the DRA established in \cite{banjac2021minimal,bauschke2023douglas}.
\begin{proposition} \label{prop:primal_shadow_infeasibility_properties}
    It holds true that \( d_{k+1} - d_k \to 0\) and any limit point \( (\overline{d}, \overline{z}) \) of \( (d_k, z_k)_{k \in \setN} \) satisfies \( (\overline{d}, \overline{z}) \in \argmin_{d \in \mc{D}, z \in \mc{C}} \norm{d - z} \).
    In addition, if \( \mc{C} \) is a polyhedron, then,
    \begin{equation} \label{eq:infeasibility_best_approximation}
	 \begin{alignedat}{2}
		d_k \to ~~  &\underset{d \in \setR^{SA}}{\mathclap{\mathrm{argmin}}} 
		\quad~ && \quad c^{\top} d
            \\
        & \overset{\hphantom{d \in \setR^{SA}}}
		{\mathclap{\mathrm{subject~to}}} 
		&& \quad d \in \mc{D} \cap( \mc{C} - v)~.
    \end{alignedat}
\end{equation}
\end{proposition}
This result is particularly interesting when \( \mc{C} \cap \mc{D} = \emptyset \).
Plainly, we learn that any limit point of \(  (d_k)_{k \in \setN} \) is an occupancy measure that is "closest", in the norm sense, to \( \mc{C} \).
Moreover, for polyhedral \( \mc{C} \), \(  (d_k)_{k \in \setN} \) converges to an optimal solution of the modified and by-construction feasible CMDP \eqref{eq:infeasibility_best_approximation}.
Pictorially, the behavior of \autoref{alg:os-cmdp} for infeasible CMDPs is shown in \autoref{subfigure:infeasibility_geometry}.
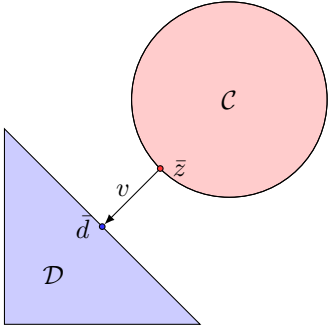
\begin{figure}[htbp]
  \centering

    \begin{tikzpicture}[scale=1.3]
		
		\path (0,0) coordinate (P1);
		\path (0, 2) coordinate (P2);
		\path (2, 0) coordinate (P3);
		\filldraw[fill=blue!20, draw=black] 
		(P1) -- (P2) -- (P3) -- cycle;
		
		\node at (0.5, 0.5) {$\mc{D}$};
			\path (1, 1) coordinate (dbar);

		\def\center{2.3}

		\filldraw[fill=red!20, draw=black] (\center, \center) circle [radius=1];
		\node at (\center, \center) {$\mc{C}$};

		\draw[name path=circle] (\center, \center) circle (1);

	    \path[name path=line] (0.75, 0.75) -- (3.31,3.31);

		\path[name intersections={of=circle and line, by={A,B}}];

		\node[draw, fill=red!80, circle, inner sep=0.75pt, label=right:$\bar{z}$] at (B) {};

            \coordinate (dbar_offset) at ($ (dbar) + (0.02, 0.02) $);
		\coordinate (B_offset) at ($ (B) + (-0.02, -0.02) $);
		\draw[latex-] (dbar_offset) -- node[pos=0.35, above] {$ v $} (B_offset);
            \node[draw, circle, fill=blue!80, inner sep=0.75pt, label=left:$\bar{d}$] at (1, 1) {};
	\end{tikzpicture}
        \caption{Asymptotic behavior of \autoref{alg:os-cmdp} for infeasible problems.
        Any limit point \( (\overline{d}, \overline{z}) \) of the iterates \( (d_k, z_k) \) minimizes the distance 
        between the sets \(\mc{C} \) and \(\mc{D}\).
        The difference of iterates \( w_{k} - w_{k+1} \) converges to the minimal displacement vector \( v \).}
        \label{subfigure:infeasibility_geometry}
\end{figure}

We stress that the results of this subsection only require \autoref{stassumption:setc}, and not \autoref{assum:feasible_constraint_qualification}.
Propositions \ref{prop:infeasibility_governing_sequence} and \ref{prop:primal_shadow_infeasibility_properties} characterize the asymptotic behavior
of the DRA irrespective of feasibility.

\subsection{Termination Criteria}
To specify meaningful termination criteria, we derive the optimality conditions of \eqref{eq:composite_form_cmdp}.
First, we introduce the auxiliary variable \( z \in \setR^{SA} \) and equivalently express \eqref{eq:composite_form_cmdp} as:
\begin{subequations} \label{eq:occupancy_cmdp_admm}
    \begin{alignat}{2}
        &\underset{ d, z \in \setR^{SA}}{\mathclap{\mathrm{minimize}}} 
        \quad~ &&  f(d) + g(z) \\
        & \overset{\hphantom{ d, z \in \setR^{SA} }}{\mathclap{\mathrm{subject~to}}} \quad~
        && d - z = 0~. \label{eq:splitting_equality_constraint}
    \end{alignat}
\end{subequations}
Then, \( (d^{\star}, z^{\star}, \nu^{\star}) \) is a primal-dual solution if and only if \cite[Th.~11.50]{rockafellar2009variational}:
\begin{subequations} \label{eq:optimality}
    \begin{align}
		& 0 \in \partial f(d^{\star}) - \nu^{\star} \label{eq:optimality_f} \\ 
		& 0 \in \partial g(z^{\star}) + \nu^{\star} \label{eq:optimality_g} \\
		& d^{\star} - z^{\star} = 0~. \label{eq:optimality_equality}
    \end{align}
\end{subequations}

Observe that, by Fermat's rule \cite[Th.~16.3]{bauschke_convex_2017}, the optimality condition of \eqref{eq:dra_regularized_mdp} reads \( 0 \in \partial f(d_k) - \nu_k \),
which implies that \eqref{eq:optimality_f} is satisfied at each iteration.
In turn, the optimality condition for \eqref{eq:dra_safety_projection} is
\(- \sigma^{-1} (z_k - d_k) \in \partial g(z_k) + \nu_k\).
By \autoref{prop:nominal_exact_convergence}, \( z_k - d_k \to 0\), therefore both \eqref{eq:optimality_g} and \eqref{eq:optimality_equality}
are satisfied in the limit.

Therefore, we terminate the algorithm with the approximate primal-dual solution \( (d_k, z_k, \nu_k) \) if:
\begin{equation}
    \begin{gathered}
        \norm{d_k - z_k}_{\infty} \leq \varepsilon_{\textrm{opt}}, \quad \text{and} \quad \\
        \max(C_i(d_k) - b_i, 0) \leq \varepsilon_{\textrm{con}}(1 + \abs{b_i}), ~ \forall i = 1, \ldots, n_c,
    \end{gathered}
\end{equation}
where \(\varepsilon_{\textrm{opt}}, \varepsilon_{\textrm{con}} > 0  \) are specified tolerances.
The first condition is in line with \eqref{eq:optimality}, while the second one explicitly controls the allowable constraint violation,
recalling that \( \mc{C} = \{ d \in \setR^{SA} \, | \, C_i(d) \leq b_i, ~ i = 1, \ldots, n_c \} \).

Conversely, characterizing a problem instance as infeasible could be achieved by testing that \( \gs_{k+1} - \gs_k \) converges to
a non-zero vector.
Unfortunately, such a criterion can be misleading since it is possible that \( \gs_{k+1} - \gs_k \) stagnates at a particular value for a number of iterations,
without having reached its limit point.
Instead, since \autoref{prop:primal_shadow_infeasibility_properties} guarantees that \( d_{k+1} - d_k \to 0 \),
we flag a problem as infeasible if \( d_{k+1} - d_k \) is small, but \( d_{k+1} \) does not satisfy the constraints up to the required tolerance.
Concretely, we check the following condition of infeasibility for some small tolerance \( \varepsilon_{\textrm{inf}} > 0\):
\begin{equation} \label{eq:infeasibility_termination}
    \begin{gathered}
        \norm{d_{k+1} - d_k}_{\infty} \leq \varepsilon_{\textrm{inf}}, \quad \text{and} \quad \\
        \max(C_i(d_k) - b_i, 0) > \varepsilon_{\textrm{con}}(1 + \abs{b_i}), ~ \forall i = 1, \ldots, n_c~.
    \end{gathered}
\end{equation}

\section{Practical Implementation} \label{sec:practical_implementation}
\subsubsection*{Inexact QRPI}%
From a practical standpoint, executing the inner loop to full convergence is computationally demanding but, as we discuss next, unnecessary. 
Instead, it suffices to solve \eqref{eq:regularized_mdp} up to some tolerance that decreases with the iteration \( k \).
Formally, let \( e_k \coloneqq \norm{d_k - \prox_{\sigma f} (\gs_k)} \) and notice that, in general, \( e_k > 0 \) unless
we execute infinitely many QRPI iterations.
If \( e_k \) decays sufficiently fast, formally \( \sum_{k = 0}^{\infty} e_k < \infty \), then
the convergence certificate of \autoref{prop:nominal_exact_convergence} remains valid \cite[Th.~7]{eckstein1992douglas}.

In principle, we could estimate an upper bound on \( e_k \), e.g., using \cite[Lemma 4.5]{Luo1992OnTC} and execute enough QRPI iterations
to ensure the necessary decrease in \( e_k \).
Yet, in practice, we have observed that warm-starting QRPI and running a small fixed number of inner iterations,
specifically \( \overline{\ell} =2 \), generates trajectories of 
\autoref{alg:os-cmdp}
that behave similarly to the case where QRPI is run to convergence (see Appendix \ref{appendix:inexact_qrpi}).
We adopt this approximate but efficient approach in our implementation.

\subsubsection*{Regularized Policy Evaluation}
The main computational bottleneck of \autoref{alg:os-cmdp} is solving the linear system \eqref{eq:inner_v_update},
for which we propose two approaches.
In the direct approach, we factorize \( (\gamma P - \Xi)^{\top} (\gamma P - \Xi) \) and solve \eqref{eq:inner_v_update}
through a computationally-inexpensive forward and backward substitution.
As shown in \autoref{lemma:trivial_nullspace}, the matrix  \( (\gamma P - \Xi)^{\top} (\gamma P - \Xi) \) is positive definite therefore we factorize it 
via a Cholesky decomposition.
Further, since this matrix does not change across iterations we only need to factorize it once during the initialization 
of \autoref{alg:os-cmdp}.
While efficient, factorizing, or even computing, \( (\gamma P - \Xi)^{\top} (\gamma P - \Xi) \) can be prohibitive when
\( P \) is very large.

Instead, in the indirect approach, we solve \eqref{eq:inner_v_update} via an iterative method, 
such as the conjugate gradient method \cite{nocedal2006numerical}, which
only requires matrix-vector product evaluations of \( (\gamma P - \Xi)^{\top} \big((\gamma P - \Xi) V \big) \). 
This approach is particularly useful when \( \gamma P - \Xi \) is large and sparse, but is generally
less efficient than the direct one, whenever the latter is applicable.
In our simulations we use the direct method.

\subsubsection*{Choosing \(\sigma\)}
The value of \( \sigma \) significantly affects the speed and transient behavior of \autoref{alg:os-cmdp}.
Intuitive design guidelines can be drawn by interpreting \( \sigma \) as a scaling constant for the objective
of problem \eqref{eq:composite_form_cmdp}.
Since the only term affected by \( \sigma \) is \( c^{\top} d \), we can conclude that large values of \( \sigma \)
prioritize cost minimization, whereas small ones prioritize constraint satisfaction.
These intuitions match our numerical experience.

\subsubsection*{Relaxation}
A popular variant of the DRA is obtained by applying relaxation to the updates of \( \gs_k \), i.e.,
replacing \eqref{eq:dra_governing_update} by
\( \gs_{k+1} = \gs_k + \omega (z_k - d_k) \),
where \( \omega \in (0, 2) \).
The relaxed iteration retains theoretical guarantees, and
numerical studies have reported faster convergence for \( \omega > 1 \) \cite{eckstein1998operator, eckstein1994parallel}.
In our implementation we use \( \omega = 1.5 \).

\section{Numerical Simulations} \label{sec:numerical_simulations}
We deploy OS-CMDP on two benchmark problems.
First, we consider Garnet MDPs \cite[Sec.~8]{bhatnagar:hal-00840470}, a well-studied class of randomly-generated MDPs,
and compare the performance of OS-CMDP against competing approaches
for CMDPs with polyhedral and more general convex constraints.
Second, we consider a grid world problem which we use to qualitatively study the behavior of our algorithm
for both feasible and infeasible problem instances.

The parameters of \autoref{alg:os-cmdp} are set to
\( \sigma = 2 \times 10^{-5}, \omega = 1.5, \; \overline{\ell}=2, \; \varepsilon_{\textrm{opt}}=10^{-5}, \; \varepsilon_{\textrm{con}}=10^{-4}, \;
\varepsilon_{\textrm{inf}}= 10^{-6} \).
All simulations ran on a Macbook Pro with 16 GB RAM.
Our code is implemented in Python using \texttt{PyTorch} \cite{paszke2019pytorch}.

\subsection{Garnet MDPs}
Garnet problems are a class of abstract, parametrized and randomly-generated 
MDPs which have often been used as benchmarks \cite{farahmand2021pid, grand2021scalable, tamar2012integrating};
see \cite[Sec.~8]{bhatnagar:hal-00840470} for a detailed description of Garnet MDPs.
An important parameter of this problem is the branching factor \( f_b \), that determines the fraction of possible
next states for each \( (s,a) \) pair or, in other words, the fraction of non-zeros entries for each row of \( P \).
In this study, we use two values of \( f_b \) corresponding to dense (\( f_b =0.5\)), and
sparse (\( f_b =0.05\)) transition matrices \( P \).
We consider MDPs with \(A = 10 \) actions and an increasing number of states \( S \)
to study the scalability of different algorithms.
The costs \( c(s,a) \) are generated randomly from a normal distribution \( \mc{N}(0, 1) \),
and we consider \( \gamma = 0.95 \).

First,  we consider linear constraints with the entries of \( \lcm \in \setR^{10 \times SA} \) drawn from \( \mc{N}(0, 1) \) and
\( b \) from \( \mc{N}(-0.2, 1) \).
We compare \autoref{alg:os-cmdp} with Gurobi \cite{gurobi}, a commercial LP solver,
and a primal-dual method (PDM) (see, e.g., \cite[Alg.~1]{paternain2022safe}), that solves an uncostrained MDP using policy iteration at each step
(for implementation details see Appendix \ref{subsubsec:implementation_pdm}).
To ensure a fair comparison, we implemented an inexact version of the primal-dual method by executing
finitely many policy iteration updates at each step.
For Gurobi, we deploy its concurrent optimizer that simultaneously executes a barrier, a primal simplex and dual simplex method, 
and terminates once any method satisfies its termination tolerances which are set to \( 10^{-4} \) for both optimality
and constraint violation.
The termination conditions for PDM are outlined in Appendix \ref{subsubsec:implementation_pdm}.
We note that, all three approaches employ the same tolerance for constraint violation but different criteria of optimality;
we chose the corresponding tolerances such that the resulting objective values
are similar.
In \autoref{tab:comparison_linear}, we report the computation times (averaged over 10 runs) and attained objective for each approach.

\begin{table*}[htbp]
    \centering
    \caption{Solver comparison for linear constraints, in terms of computation times and objective value, reported in parantheses. 
    The best performing solver is shown in bold, for each metric and setting.}
    
    \setlength{\aboverulesep}{0pt}
    \setlength{\belowrulesep}{0pt}
    \setlength{\extrarowheight}{.75ex}
    \resizebox{\linewidth}{!}{
    \begin{tabular}{@{}ccccccc@{}}
        \toprule
        & \multicolumn{3}{c}{ \( f_b = 0.05 \)} & \multicolumn{3}{c}{\( f_b=0.5 \)} 
        \\
         \cmidrule(lr){2-4} \cmidrule(l){5-7}
        Solver & OS-CMDP & Gurobi & PDM & OS-CMDP & Gurobi & PDM \\ \midrule
        \( S = 100\hphantom{0} \) & \cellcolor{gray!25} $0.48s~(-1.06)$ & $\textbf{0.05s}~(\textbf{-1.12})$ & $0.89s~(-1.11)$ & \cellcolor{gray!25} $0.39s~(-1.05)$ & $\textbf{0.07s}~(\textbf{-1.10})$ & $0.69s~(-1.09)$ \\
        \( S = 1000\hphantom{} \) &\cellcolor{gray!25} $\textbf{1.00s}~(-1.07)$ & $1.80s~(\textbf{-1.08})$ & $3.64s~(-1.04)$ & \cellcolor{gray!25} $\textbf{0.63s}~(-1.08)$ & $2.96s~(\textbf{-1.11})$ & $3.22s~(-1.08)$ \\
        \( S = 3000\hphantom{} \) & \cellcolor{gray!25}$ \textbf{11.52s}~(-1.06)$ & $20.13s~(\textbf{-1.07})$ & $70.66s~(-1.02)$ & \cellcolor{gray!25} $\textbf{4.98s}~(-1.08)$ & $118.15s~(\textbf{-1.09})$ & $30.54s~(-1.04)$ \\
        \( S = 5000 \) & \cellcolor{gray!25} $\textbf{11.89s}~(-1.03) $ & $71.98s~(\textbf{-1.07})$ & $88.37s~(-1.01)$ & \cellcolor{gray!25} $\textbf{10.58s}~(-1.05)$ & $756.65s~(\textbf{-1.08})$ & $86.28s~(-1.02)$ \\
         \bottomrule
    \end{tabular}
    \label{tab:comparison_linear}
    }
\end{table*}
We observe that in all cases OS-CMDP performs on par with other solvers in terms of objective value.
In particular, it typically yields a better objective than PDM and slightly worse than Gurobi (in the
order of \(1-5\% \) worse).
This is a consequence of the fact that Gurobi is suited to finding high-accuracy solutions.
In terms of computation time, OS-CMDP significantly outperforms other approaches, especially
as the problem dimension increases.
We observe that OS-CMDP performs similarly regardless of density,
whereas Gurobi is able to exploit sparsity in the problem.
We note that for OS-CMDP we take into account both the execution time and the setup time, i.e.,
the computation and factorization of \( (\gamma P - \Xi)^{\top} (\gamma P - \Xi) \).

Next, we consider \( \ell_2 \)-norm constraints as in Example 2, with randomly generated
\( \hat{d} \in \mc{D} \) and \( \varepsilon = 0.2 \).
We compare \autoref{alg:os-cmdp} against Gurobi and SCS \cite{o2016conic}, a first-order solver
for conic programs.
We use SCS instead of a primal-dual scheme since the latter is not directly applicable for
non-linear constraints.
We set the optimality and feasibility tolerance of SCS to \( 10^{-4} \).
Results are presented in \autoref{tab:comparison_l2}.

\begin{table*}[htbp]
    \centering
    \caption{Solver comparison for \( \ell_2\)-constraints, in terms of computation times and objective value, reported in parantheses. 
    OOM indicates that the solver ran out of memory.
    The best performing solver is shown in bold, for each metric and setting.}
    
    \setlength{\aboverulesep}{0pt}
    \setlength{\belowrulesep}{0pt}
    \setlength{\extrarowheight}{.75ex}
    \resizebox{\linewidth}{!}{
    \begin{tabular}{@{}ccccccc@{}}
        \toprule
        & \multicolumn{3}{c}{ \( f_b = 0.05 \)} & \multicolumn{3}{c}{\( f_b=0.5 \)} 
        \\
         \cmidrule(lr){2-4} \cmidrule(l){5-7}
        Solver & OS-CMDP & Gurobi & SCS & OS-CMDP & Gurobi & SCS \\ \midrule
        \( S = 100\hphantom{0} \) & \cellcolor{gray!25} $\textbf{0.01s}~(\textbf{-0.55})$ & $0.03s~(\textbf{-0.55})$ & $0.016s~(\textbf{-0.55})$ 
        & \cellcolor{gray!25} $\textbf{0.01s}~(\textbf{-0.58})$ & $0.68s~(\textbf{-0.58})$ & $0.06s~(\textbf{-0.58})$ \\
        \( S = 1000\hphantom{} \) &\cellcolor{gray!25} $\textbf{0.38s}~(\textbf{-1.39})$ & $2.72s~(\textbf{-1.39})$ & $1.29s~(\textbf{-1.39})$ 
        & \cellcolor{gray!25} $\textbf{0.39s}~(-1.45)$ & $224.02s~(\textbf{-1.46})$ & $8.77s~(\textbf{-1.46})$ \\
        \( S = 3000\hphantom{} \) & \cellcolor{gray!25} $\textbf{3.19s}~(-1.49)$ & $44.61s~(-1.53)$ & $39.80s~(\textbf{-1.53})$ 
        & \cellcolor{gray!25} $\textbf{3.18s}~(-1.51)$ & $\text{OOM}$ & $194.85~(\textbf{-1.55})$ \\
        \( S = 5000 \) & \cellcolor{gray!25} $\textbf{10.44s}~(-1.51)$ & $97.33s~(-1.52)$ & $185.73s~(\textbf{-1.53})$ & \cellcolor{gray!25} $\textbf{10.45s}~(-1.52)$ & $\text{OOM}$ & $984.81s~(\textbf{-1.54})$ \\
         \bottomrule
    \end{tabular}
    \label{tab:comparison_l2}
    }
\end{table*}
Again, we note that OS-CMDP retrieves a solution significantly faster than competing methods,
though with a slightly worse objective value.
Obtaining a high-accuracy solution can be challenging for OS-CMDP, since it is a first-order method, 
but its strength lies in rapidly solving the problem up to medium accuracy.
To highlight this, we remark that if we fix the computational budget to the time that OS-CMDP needs to converge,
then other approaches yield highly-suboptimal solutions, or no solutions at all.

\subsection{Grid world}
Next, we study the qualitative behavior of \autoref{alg:os-cmdp} for both feasible and infeasible constraints by
deploying it on a grid world problem,
wherein an agent tries to navigate from a start location to a destination while avoiding obstacles and
minimizing path length.
The environment of the agent is represented as a grid whose cells correspond to states, some of which
are occupied by obstacles.
The agent can play the actions \texttt{\{up, down, left, right\}} that induce a transition
to the corresponding neighboring cell\footnotemark with probability \( 1 - \delta\), where \( \delta \in [0,1) \).
\footnotetext{For simplicity, we assume that executing actions 
that lead outside the grid do not affect the agent's position.}
Conversely, with probability \(\delta\) the agent ends up in a (uniformly) randomly-selected neighboring cell.
This problem is inspired from case studies in \cite{khairy2020gradientaware, chow2018risk, tessler2018reward}.
Since this benchmark features very sparse \( P \), we will focus on qualitative characteristics of
\autoref{alg:os-cmdp} rather than computational performance.

In our simulations, we consider a \( 25 \times 25 \) grid, consisting of \( 625 \) states and \( 45 \) obstacles,
and we set \( \gamma = 0.99\), and \( \delta = 0.05 \).
The top left and bottom right cells are the respective start and destination of the agent.
To minimize path length, we set the cost \( c \) to 1 for all non-destination states
and to 0 for the destination.
We enforce obstacle avoidance as a linear constraint \( \co^{\top} d \leq \bo \),
where \( \co(s,a) \) is equal to 1 if state \( s \) corresponds to an obstacle, and 0 otherwise.
The threshold \( \bo > 0\) prescribes the allowable (discounted) probability of collision.

We note that the stage cost \( c \) only encourages, but does not force, the agent to reach its destination.
To achieve that, we impose the additional constraint \( c^{\top} d \leq \bp \)
for \( \bp \in [0, 1] \).
Intuitively, decreasing \( \bp \) increases the probability of reaching the destination
and decreases the resulting path length.
We remark that \( c^{\top} d = 1 \) if the agent almost surely does not reach the destination,
whereas if the agent disregards the obstacles and only minimizes path length
we compute that \( c^{\top} d = 0.396 \) by solving 
the same MDP without constraints.

We will study the occupancy measure returned by \autoref{alg:os-cmdp} for various choices of
\( \bo \) and \( \bp \).
We remark that certain combinations of \( \bo \) and \( \bp \) render the problem numerically challenging,
for which we utilize smaller tolerances and scaling parameter \( \sigma \).

In Figures \ref{subfig:grid_feasible_loose} and \ref{subfig:grid_feasible_tight},
we consider \( (\bp, \bo) = (0.9, 10^{-3}) \) and \( (\bp, \bo) = (0.9, 2\times10^{-4}) \), respectively.
Colors represent the state-marginal occupancy measure of each cell \( s \), i.e., \( \sum_{a} d(s,a) \),
and arrows indicate that the corresponding action is chosen with non-zero probability, formally \( d(s,a) > 10^{-10} \).
Under both choices of thresholds, the problem is feasible and for smaller \( \bo \) the agent sacrifices path length
in favor of obstacle avoidance.

\begin{figure*}
    \centering
    \begin{subfigure}{0.44\textwidth}
        \centering
        \includegraphics[height=6.1cm]{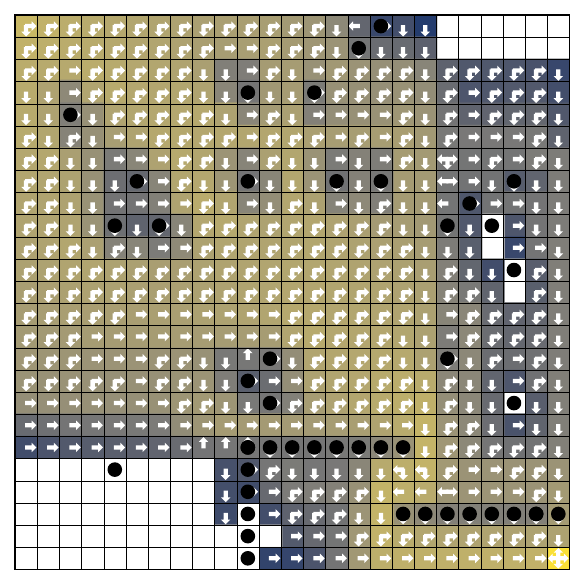}
        \caption{\( (\bp, \bo) = (0.9, 10^{-3}) \)}
        \label{subfig:grid_feasible_loose}
    \end{subfigure}
    \hfill
    \begin{subfigure}{0.55\textwidth}
        \centering
        \includegraphics[height=6.1cm]{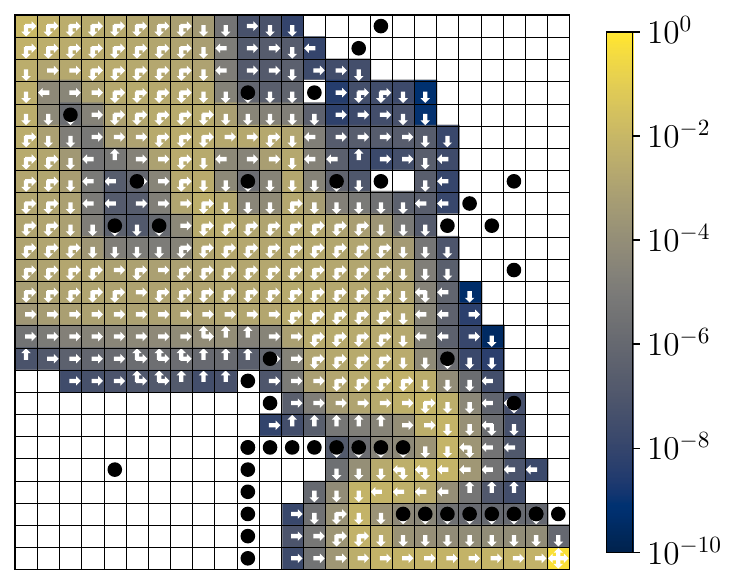}
        \caption{\( (\bp, \bo) = (0.9, 2\times10^{-4}) \)}
        \label{subfig:grid_feasible_tight}
    \end{subfigure}
    \caption{State marginal occupancy measure (in color) and policy (as arrows) for two feasible choices
    of \( (\bp, \bo) \). 
    The thresholds \( \bp \) and \( \bo\) correspond to the constraints of reaching the destination and
    avoiding collisions, respectively.
    Black circles indicate obstacles.
    Values below \( 10^{-10} \) are shown in white.}
    \label{fig:grid_feasible}
\end{figure*}

In Figure \ref{subfig:grid_infeasible_loose}, we consider \( (\bp, \bo) = (0.9, 2\times10^{-5}) \), 
that renders the problem infeasible, i.e., the agent cannot reach the destination without collision.
The resulting policy simply tries to avoid the obstacles but, interestingly, if
the agents ends up close to the destination then it will try to reach it.
In Figure \ref{subfig:grid_infeasible_tight}, we set \( (\bp, \bo) = (0.6, 2\times10^{-5}) \) which is again infeasible,
but tighter in terms of the path length constraint.
Now, the resulting policy appears to randomly choose between two "strategies".
One leads the agent to the destination via a path that is not sufficiently safe,
and the other makes the agent wander in a safe and obstacle-free part of the grid.
Our results highlight the ability of OS-CMDP to extract 
meaningful policies even from infeasible specifications,
and indicate how individual constraints affect feasibility.

We remark that regions where \( d \) seems inconsistent with the dynamics, e.g., arrows from coloured cells
pointing into white cells at the top right of Figure \ref{subfig:grid_infeasible_tight},
are numerical artifacts from approximately enforcing \( d \in \mc{D} \).
\begin{figure*}
    \centering
    \begin{subfigure}{0.44\textwidth}
        \centering
        \includegraphics[height=6.1cm]{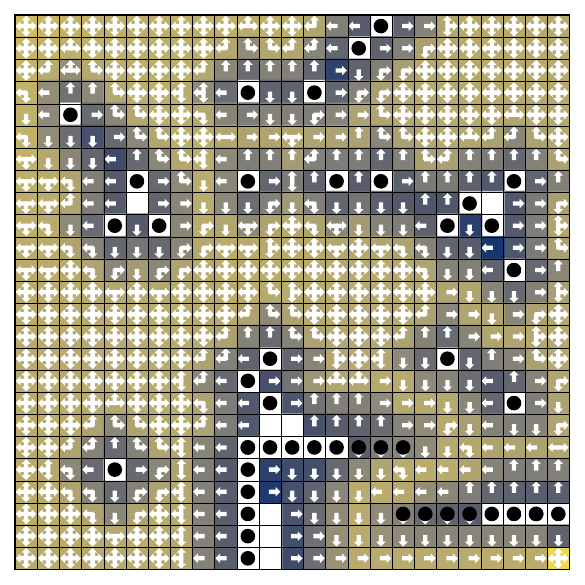}
        \caption{\( (\bp, \bo) = (0.9, 2\times10^{-5}) \)}
        \label{subfig:grid_infeasible_loose}
    \end{subfigure}
    \hfill
    \begin{subfigure}{0.55\textwidth}
        \centering
        \includegraphics[height=6.1cm]{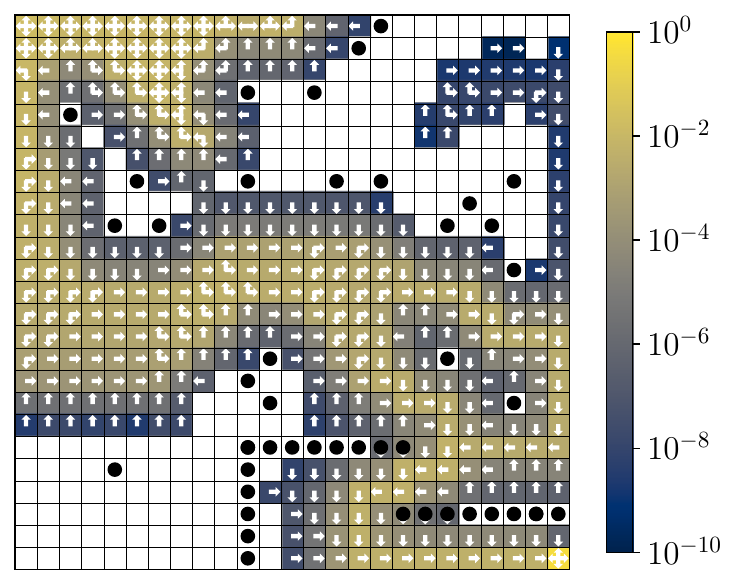}
        \caption{\( (\bp, \bo) = (0.6, 2\times10^{-5}) \)}
        \label{subfig:grid_infeasible_tight}
    \end{subfigure}
    \caption{State marginal occupancy measure (in color) and policy (as arrows) for two infeasible choices
    of \( (\bp, \bo) \). 
    The thresholds \( \bp \) and \( \bo\) correspond to the constraints of reaching the destination and
    avoiding collisions, respectively.
    Black circles indicate obstacles. Values below \( 10^{-10} \) are shown in white.}
    \label{fig:grid_infeasible}
\end{figure*}

\section{Conclusion}
In this work, we developed a first-order algorithm for MDPs with convex constraints by deploying
the Douglas-Rachford splitting.
The main challenge in our approach is solving a quadratically-regularized MDP,
for which we proposed an efficient scheme inspired by policy iteration.
We demonstrated that our algorithm compares favorably against existing approaches on a benchmark problem,
and is able to gracefully handle infeasible specifications.
The main limitation of our work is that it requires perfect knowledge of the underlying MDP and relies
on a direct parametrization of the occupancy measure.
Future work includes extending our algorithm to the safe RL setting and employing function approximation. 

\appendices
\section*{Appendix}
\renewcommand\thesubsection{\Alph{subsection}}
\subsection{Derivations  of \autoref{sec:algorithm_design}} \label{app:derivation_of_qrpi}
\paragraph*{Derivation of \eqref{eq:regularized_mdp}}
To derive the dual problem of the regularized MDP \eqref{eq:regularized_mdp},
we first write the Lagrangian:
\begin{equation} \label{eq:lagragian_reg_mdp}
    \begin{aligned}
    \mc{L}(d; V, \varphi) = & (c + \gamma P V - \Xi V )^{\top} d + (1 - \gamma) \rho^{\top} V \\
    & + \frac{1}{2\sigma} \norm{d - \gs_k}^2 - \varphi^{\top} d~,
    \end{aligned}
\end{equation}
where \( \varphi \geq 0 \). We derive the dual function by minimizing \(\mc{L}\) with respect to \( d \).
The Lagrangian is a strongly convex function of \( d \), so we find the
unique minimizer by setting \( \nabla_d \mc{L}(d; V, \varphi) = 0 \), and solving for \(d\) yields:
\begin{equation} \label{eq:lagrangian_minimizer_d}
    d = \gs_k + \sigma( - c - \gamma P V + \Xi V  + \varphi)~.
\end{equation}
Substituting \eqref{eq:lagrangian_minimizer_d} onto \eqref{eq:lagragian_reg_mdp} yields
the dual function \( \kappa \), as shown in \eqref{eq:dual_regularized_mdp}. \( \hfill \square\)
\begin{lemma} \label{lemma:trivial_nullspace}
    The matrix \( \gamma P - \Xi \in \setR^{SA \times S} \) has trivial nullspace.
\end{lemma} 
\begin{proof}
    The matrices \( \Xi \) and \( P \) are given by:
    \begin{align}
        \Xi & \coloneqq \begin{bmatrix}
            I_S \\ 
            \vdots \\
            I_S
        \end{bmatrix}, ~
        P \coloneqq \begin{bmatrix}
            (P(s' \, | \, s=1, a=1))_{s' \in \mc{S}} \\
            (P(s' \, | \, s=2, a=1))_{s' \in \mc{S}} \\
            \vdots \\
            (P(s' \, | \, s=S, a=A))_{s' \in \mc{S}} 
        \end{bmatrix}.
    \end{align}
    Then, notice that 
    \begin{equation}
        \gamma P - \Xi = \begin{bmatrix}
            \gamma P_{a=1} - I_S \\ \vdots \\
            \gamma P_{a=A} - I_S
        \end{bmatrix}
    \end{equation}
    where \( P_{a=i} \) denotes the transition matrix induced by the policy
    that always chooses action \( i \in \mc{A} \).
    Standard dynamic programming arguments \cite[Th.~6.1.1]{puterman1990markov} guarantee that, for any \( i \in \mc{A} \),
    the matrix \( \gamma P_{a=i} - I_S \) is invertible, hence \( \gamma P - \Xi \) has trivial nullspace.
\end{proof}
\paragraph*{Derivation of \eqref{eq:dual_occupancy_optimality}}
The maximizer of the dual function \( \kappa \) for a fixed \( V^{\star} \) can be expressed as follows:
\begin{align*}
    & \argmax_{\varphi \geq 0} ~- \frac{\sigma}{2} \norm{c + \gamma P V^{\star} - \Xi V^{\star} - \varphi}^2 
    \\ & ~~~~~~~ + \gs_k^{\top} (c + \gamma P V^{\star} - \Xi V^{\star} - \varphi)
    + (1 - \gamma) \rho^{\top} V^{\star} \\
     \overset{(a)}{=} & \argmax_{\varphi \geq 0} - \norm{c + \gamma P V^{\star} - \Xi V^{\star} - \varphi}^2
     \\ & ~~~~~~~ + 2 \gs_k^{\top}/\sigma (c + \gamma P V - \Xi V - \varphi) - \norm{\gs_k/\sigma}^2 \\
     \overset{(b)}{=} & \argmax_{\varphi \geq 0 } - \norm{\varphi - (c + \gamma P V - \Xi V - \gs_k/\sigma)}^2 \\
     \overset{(c)}{=} &  \max(c + \gamma P V - \Xi V - \gs_k/\sigma, 0),
\end{align*}
where \((a)\) is by multiplying with \( 2/\sigma > 0\) and adding/removing terms that do not depend on \( \varphi \),
\((b)\) by factorizing the square,
and \((c)\) by noting that the problem in \((b)\) is an orthogonal projection onto the non-negative orthant.
\( \hfill \square \)

\subsection{Proofs of \autoref{sec:convergence_analysis}} \label{app:convergence_proofs}
\paragraph*{Proof of \autoref{prop:qrpi_convergence}}
To prove the claim, we will invoke \cite[Th.~2.1]{Luo1992OnTC}, hence we proceed to verify that its preconditions hold.
For ease of reference, in this proof we will use the notation in \cite{Luo1992OnTC}, i.e.,
let \( g(\cdot) \coloneqq \norm{c+\cdot}_2^2, E \coloneqq \begin{bmatrix}
    \gamma P - \Xi & -I
\end{bmatrix},
b \coloneqq \begin{bmatrix}
    (\gamma P - \Xi)^{\top} \gs_k + (1- \gamma \rho) & - \gs_k
\end{bmatrix}\).
Observe that, since \( \gamma P - \Xi\) and \( -I \) are full column rank, then \( E \) has no zero column.

Next, we will establish that Assumption A1(a) in \cite{Luo1992OnTC} holds.
Observe that, strong duality holds for \eqref{eq:regularized_mdp}, by virtue of Slater's condition since all the constraints of the convex program \eqref{eq:regularized_mdp} are affine,
and the problem is feasible \cite[\S 5.2.3]{boyd2004convex}.
Feasibility is guaranteed by the fact that any stationary Markov policy generates an
occupancy measure that satisfies \eqref{eq:regularized_mdp_flow_constraint}, and \eqref{eq:regularized_mdp_positivity}\cite[Th.\ 3.2]{altman2021constrained}.
The primal problem \eqref{eq:regularized_mdp} is strongly convex, hence the optimal value is attained and finite.
As such, the dual problem \eqref{eq:dual_regularized_mdp} admits a solution, which
verifies Assumption A1(a) in \cite{Luo1992OnTC}.

Points (b) and (c) in Assumption A1 directly follow from \( g \) being strongly convex.
Finally, notice that the iteration \eqref{eq:block_coordinate_ascent}, coincides with the 
Coordinate Descent Method given in \cite[Eq.~(8),(9)]{Luo1992OnTC} under a cyclic update rule 
and employing block coordinate updates (which does not compromise the convergence guarantees, as
discussed in \cite[Sec.~6]{Luo1992OnTC}).
The result follows by invoking \cite[Th.~2.1]{Luo1992OnTC}.
\( \hfill \square \)

\paragraph*{Proof of \autoref{prop:nominal_exact_convergence}}
Under the assumption that the inner loop of \autoref{alg:os-cmdp} is executed to full convergence,
\autoref{alg:os-cmdp} coincides with the DRA \eqref{eq:dra_for_cmdps} applied to problem \eqref{eq:occupancy_cmdp}.
To prove the claim we will invoke \cite[Cor.~28.3]{bauschke_convex_2017}.
The functions \( f \) and \( g \) are proper, closed, and convex.
Further, \autoref{assum:feasible_constraint_qualification} ensures
that \( 0 \in \partial f + \partial g \) by virtue of \cite[Cor.~27.6]{bauschke_convex_2017}.
Having established its preconditions, \cite[Cor.~28.3]{bauschke_convex_2017} yields the claim.
\( \hfill \square \)

\paragraph*{Proof of \autoref{prop:infeasibility_governing_sequence}}
By Fact 2.1 and Proposition 3.2 of \cite{banjac2021minimal}, we know that
\( \gs_{k} - \gs_{k+1} \to v \coloneq \proj{\overline{\dom f - \dom g} \cap \overline{\dom f^* + \dom g^*}}(0) \), 
and \( v = v_P + v_D\) where
\( v_P \coloneqq \proj{\overline{\dom f - \dom g}}(0) \) and
\( v_D \coloneqq \proj{\overline{\dom f^* + \dom g^*}}(0) \).
To establish the claim, we will show that for the considered \( f \) and \( g \)
it holds \( v_D = 0\) and \( v_P = \argmin_{\beta \in \mc{D} - \mc{C}} \norm{\beta} \).

First, to show that \( v_D = 0 \) we proceed along the lines of \cite[Example 6.2]{bauschke2023douglas}.
We have that \( \mc{D} \) is a compact polyhedron \cite[Th.~3.2]{altman2021constrained},
hence \( \partial f \) is a maximally monotone operator \cite[Th.~20.25]{bauschke_convex_2017}
with a bounded domain \( \dom \partial f \subseteq \mc{D} \).
Then, invoking \cite[Cor.~21.25]{bauschke_convex_2017} we deduce that \( \ran \partial f = \setR^{SA} \)
and, since, \( \ran \partial f = \dom \partial f^{*} \subseteq \dom f^{*} \) \cite[Cor.~16.30]{bauschke_convex_2017},
we learn that \( \dom f^{*} = \setR^{SA} \).
Further, \( g^* \) is equal to the support function of \( \mc{C} \), and since \( \mc{C} \neq \emptyset \)
we have \( 0 \in \dom g^* \neq \emptyset \).
Therefore, we conclude that \( 0 \in \dom f^{\star} + \dom g^{\star} \),
which immediately implies \( v_D = 0 \).

Next, we observe that \( \dom f = \mc{D} \) and \( \dom g = \mc{C} \).
The set \( \mc{D} \) is compact and \( \mc{C} \) is closed, therefore
\( \overline{\mc{D} - \mc{C}} = \overline{\mc{D}} - \overline{\mc{C}} = \mc{D} - \mc{C}  \).
The claim follows by noting that
\( v_P = \proj{\mc{D} - \mc{C}}(0) = \argmin_{\beta \in \mc{D} - \mc{C}} \norm{\beta} \).
\( \hfill \square\)

\begin{lemma} \label{lemma:strongly_separating_hyperplane}
    Assume that \( \mc{C} \cap \mc{D} = \emptyset \). Then, the minimal displacement vector \( v \) 
    is the normal vector of a hyperplane
    that strongly separates \(\mc{C} \) and \( \mc{D} \). Formally,
    \begin{equation}
        \min_{d \in \mc{D}} v^{\top} d > \sup_{z \in \mc{C}} v^{\top} z~.
    \end{equation}
\end{lemma}
\begin{proof}
    By \autoref{prop:primal_shadow_infeasibility_properties} we know that any limit point
    \( (\overline{d}, \overline{z}) \in \argmin_{d \in \mc{D}, z \in \mc{C}}\norm{d - z} \).
    The optimality conditions of the latter problem are:
    \begin{equation} \label{eq:normal_cone_conditions}
        \left\{ 
        \begin{aligned}
            & 0 \in \overline{d} - \overline{z} + \ncone_{\mc{D}}(\overline{d}) \\
            & 0 \in \overline{z} - \overline{d} + \ncone_{\mc{C}}(\overline{z})
        \end{aligned}
        \right. \iff
        \left\{
        \begin{aligned}
            - & v \in \ncone_{\mc{D}}(\overline{d}) \\
            & v \in \ncone_{\mc{C}}(\overline{z})
        \end{aligned}
        \right.,
    \end{equation}
    where the equivalence follows by definition of \( \overline{d} \) and \( \overline{z} \).
    The conditions \eqref{eq:normal_cone_conditions} imply that
    \begin{subequations}
        \begin{align}
            & \inf_{d \in \mc{D}} v^{\top} d = v^{\top} \overline{d} \\
            & \sup_{z \in \mc{C}} v^{\top} z = v^{\top} \overline{z}~.
        \end{align}
    \end{subequations}  
    Subtracting the two equations yields
    \begin{equation}
        \inf_{d \in \mc{D}} v^{\top} d - \sup_{z \in \mc{C}} v^{\top} z = v^{\top} (\overline{d} - \overline{z}) = \norm{v}^2,
    \end{equation}
    which implies that \( \inf_{d \in \mc{D}} v^{\top} d > \sup_{z \in \mc{C}} v^{\top} z \), 
    since \( v \neq 0 \) due to \( \mc{C} \cap \mc{D} = \emptyset \) and \autoref{prop:infeasibility_governing_sequence}.
    We conclude the proof by noting that \( \mc{D} \) is compact and, hence, the infimum is attained.
\end{proof}
\paragraph*{Proof of \autoref{prop:primal_shadow_infeasibility_properties}}
Combining \cite[Th.~3.4]{banjac2021minimal} with the fact that \( v_D = 0\) implies
that \( d_{k+1} - d_k \to 0\).
The sequences \( (d_k)_{k \in \setN}, (z_k)_{k \in \setN} \) are bounded 
by virtue of \cite[Cor.~5.1]{bauschke2023douglas}, 
thus their respective set of limit points is non-empty.
By the update rule \eqref{eq:dra_governing_update} it holds
\( \gs_k - \gs_{k+1} = d_k - z_k \) and, hence,
\( v = \overline{d} - \overline{z} \)
for any convergent subsequence of \( (d_k, z_k) \)
with limit \( (\overline{d}, \overline{z}) \).
Recalling that \( v = \argmin_{\beta \in \mc{D} - \mc{C}} \norm{\beta} \), we can readily observe
that \( (\overline{d}, \overline{z}) \in \argmin_{d \in \mc{D}, z \in \mc{C}} \norm{d - z} \)
which establishes the first claim.

For the second claim, i.e., for polyhedral \( \mc{C} \),
we will invoke \cite[Th.~6.1(i)]{bauschke2023douglas}, and so we 
verify that its assumptions hold in our setting.
First, in the proof of \autoref{prop:infeasibility_governing_sequence} we have shown that \( 0 \in \dom f^{\star} + g^{\star} \).
Additionally, we note that \cite[Example 6.1(i)]{bauschke2023douglas} guarantees 
that \( Z \neq \emptyset \), hence \( v \in \ran (\id - T) \),
where \( Z \) and \( T \) are defined in equations (14) and (3) in \cite{bauschke2023douglas}, respectively.

To show that \( \zer(\partial f + \partial g (\cdot - v)) \neq \emptyset \), we recall that
\( v \in \dom f - \dom g \), hence
\( 0 \in \dom f - (v + \dom g) \implies \dom f \cap \dom g(\cdot - v) \neq \emptyset \).
Thus, invoking \cite[Cor.~27.3(i),(ii)]{bauschke_convex_2017} (under its precondition
(c)) we obtain \( \zer(\partial f + \partial g (\cdot - v)) \neq \emptyset \),
since \( \argmin_x f(x) + g(x - v) \) is non-empty (because \( f + g \) is a proper, closed, convex
function with a compact domain \cite[Th.~11.10]{bauschke_convex_2017}).
We have established all the assumptions of \cite[Th.~6.1(i)]{bauschke2023douglas} and the
claim follows.
\( \hfill \square\)

\subsection{Inexact QRPI} \label{appendix:inexact_qrpi}
In this section, we provide numerical verification that the inexact version of QRPI with
\( \overline{\ell} = 2 \) performs similarly to the exact version.
For the exact version, we use \( \overline{\ell} = 100 \) to approximate an
infinite number of QRPI iterations.
We compare the performance of exact and inexact QRPI in terms of three metrics: objective value,
constraint violation, and compatibility with dynamics;
we recall that if we execute finitely many QRPI iterations
it is not guaranteed that \( d_k \in \mc{D} \), i.e.,
\( d_k \) might not be a valid occupancy measure that is compatible with the dynamics.

For the objective value, we use the relative suboptimality of the iterates,
namely \( (c^{\top} d_k - c^{\top} d^{\star})/\abs{c^{\top} d^{\star}} \)
where \( d^{\star} \) is an optimal solution of \eqref{eq:occupancy_cmdp}
computed to high accuracy.
For constraint violation, we use \( \max_i \{ \max(C_i(d_k) - b_i, 0)\} \).
To measure compatibility with dynamics, we use
\( \norm{\Xi^{\top} d - (1 -\gamma) \rho - \gamma P^{\top} d}_{\infty} \) to
quantify the violation of the equality constraint in \eqref{eq:occupancy_measure_set}.
Note that, non-negativity of \( d_k \) is satisfied by design due to the update \eqref{eq:inner_d_update}.
\begin{figure*}
    \centering
    \hspace*{0.04\textwidth}
    \begin{subfigure}{0.45\textwidth}
        \centering
        \includegraphics[width=\textwidth]{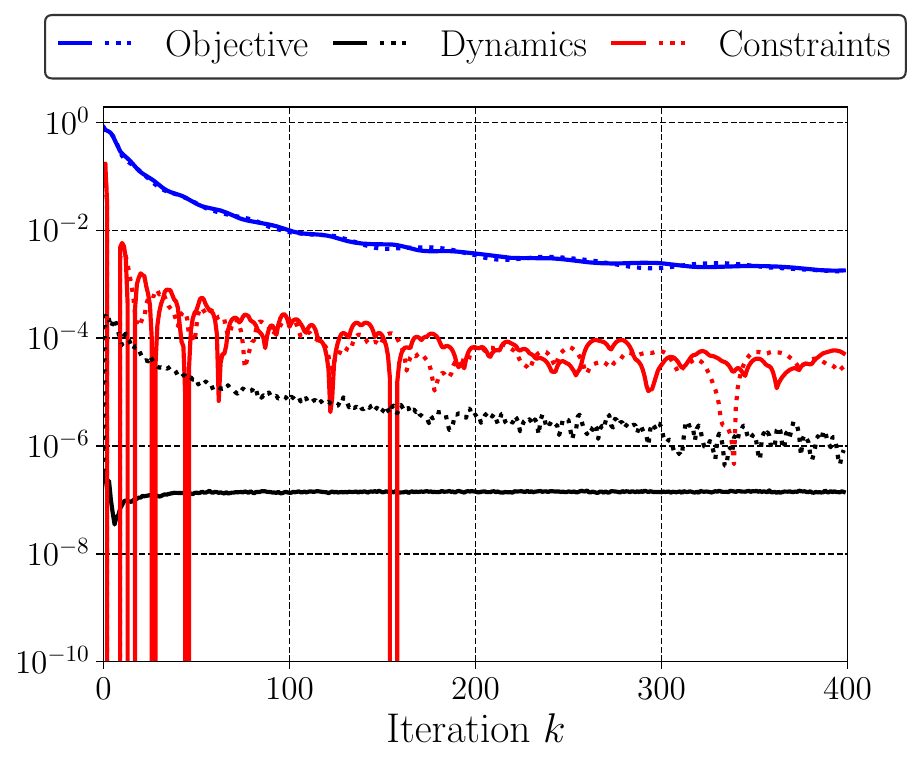}
        \caption{Garnet}
        \label{subfig:inner_comparison_garnet}
    \end{subfigure}
    \begin{subfigure}{0.45\textwidth}
        \centering
        \includegraphics[width=\textwidth]{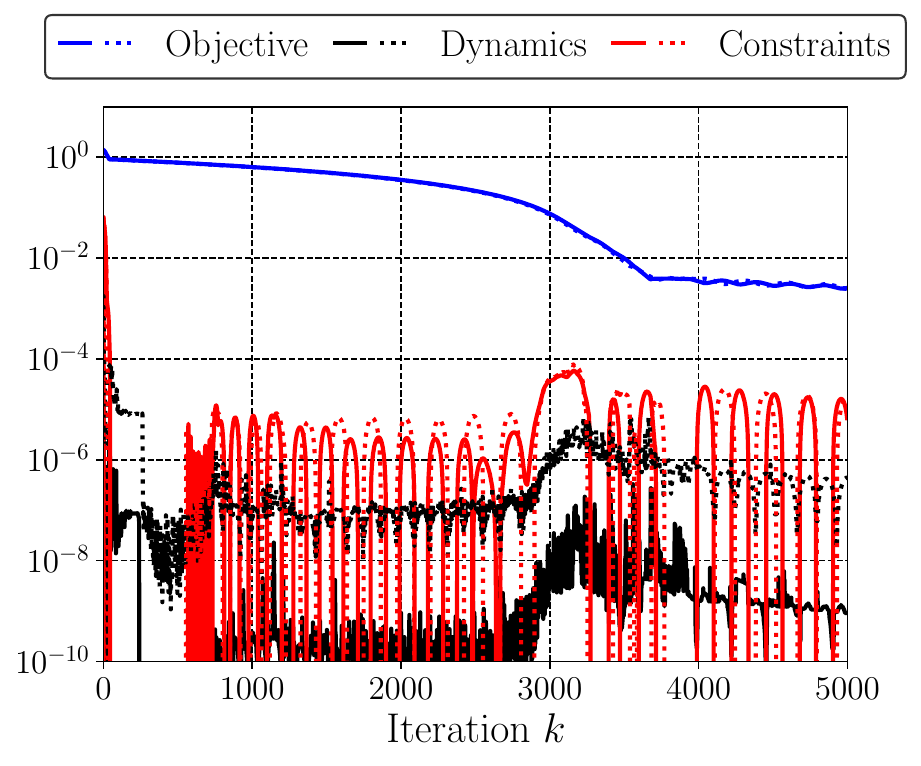}
        \caption{Grid world}
        \label{subfig:inner_comparison_grid_world}
    \end{subfigure}
    \hspace*{0.04\textwidth}
    \caption{Performance comparison of \autoref{alg:os-cmdp} with exact (continuous line) and inexact (dotted line) QRPI in the inner loop.}
    \label{fig:inner_comparison}
\end{figure*}

In Figures \ref{subfig:inner_comparison_garnet} and \ref{subfig:inner_comparison_grid_world}
we plot the three metrics as a function of the iteration \( k \) for an instance of the linearly-constrained Garnet problem
and the grid world, respectively, as introduced in \autoref{sec:numerical_simulations}.
The continuous line corresponds to the exact variant while the dotted line indicates the inexact one.

We observe that the performance of exact and inexact QRPI closely match in terms of objective
and constraint violation.
In terms of dynamics, the exact version consistently attains a high compatibility, whereas
the inexact variant is initially incompatible but gradually improves.
We stress that this transient incompatibility is a minor issue, since we are only
interested that \( d_k \in \mc{D} \) for the final iterate.
Practically,
once the termination criteria are met, one can always run an additional iteration
of \autoref{alg:os-cmdp} and execute enough QRPI iterations such
that \( d_k \in \mc{D} \) up to the desired accuracy.
In our implementation we have included this safeguard.

\subsection{Implementation of PDM} \label{subsubsec:implementation_pdm}
In the case of linear constraints \( \mc{C} = \{ d \in \setR^{SA} \,|\, \lcm d \leq b \} \)
we can solve \eqref{eq:occupancy_cmdp} via primal-dual schemes that exploit
the structure of the Lagrangian:
\begin{align*}
    \mc{L}(d; \lambda) & = c^{\top} d + \indic{\mc{D}}(d) + \lambda^{\top}(\lcm d - b) \\
    & = (c + \lcm^{\top} \lambda)^{\top} d + \indic{\mc{D}}(d) -\lambda^{\top} b~.
\end{align*}
Then, we seek a solution of the saddle-point problem:
\begin{equation*}
    \max_{\lambda \geq 0}~ \min_{d \in \mc{D}} ~ (c + \lcm^{\top} \lambda)^{\top} d -\lambda^{\top} b~.
\end{equation*}
Exploiting the fact that the inner minimization is an unconstrained MDP with cost \( c+\lcm^{\top} \lambda\),
we can deploy a dual ascent scheme.
Specifically, let \( \lambda_0 \in \setRp^{n_c} \) and iterate:
\begin{subequations} \label{eq:dual_ascent}
    \begin{empheq}[left = (\forall k \in \setN) \quad \empheqlfloor]{align}
    d_k & \leftarrow \argmin_{d \in \mc{D}} (c + \lcm^{\top} \lambda_k)^{\top} d
    \label{eq:dual_ascent_modified_mdp} \\
    \lambda_{k+1} & \leftarrow \max(\lambda_k + \alpha_k (\lcm d_k - b), 0)
    \label{eq:dual_ascent_projected_ascent}
    \end{empheq}
\end{subequations}
where \( \alpha_k > 0 \) is a step-size that satisfies
\(\alpha_k \to 0 \) and \( \sum_{k=0}^{\infty} \alpha_k = \infty \).
For our benchmarks, we tested different step sizes and we used the one that achieved fastest convergence,
specifically \( \alpha_k = 10/(k+1) \).
By virtue of \cite[Th.~6]{anstreicher2009two}, the sequence \( (\frac{1}{k}\sum_{i=0}^{k} d_i, \lambda_k)_{k \in \setN} \) converges to 
primal-dual solution of \eqref{eq:occupancy_cmdp}.
The unconstrained MDP \eqref{eq:dual_ascent_modified_mdp} can be solved using any DP technique.
In our implementation, we use policy iteration.
To ensure a fair comparison with OS-CMDP, we only run a finite number of policy iterations and, specifically we found
that 2 iterations produce a good tradeoff between computational performance and accuracy.
Further, similarly to OS-CMDP, we warm-start policy iteration for \eqref{eq:dual_ascent_modified_mdp} at iteration \(k\)
using the approximately optimal value function computed at \(k-1\).
We terminate \eqref{eq:dual_ascent} when \(\norm{\lambda_{k+1} - \lambda_k}_{\infty} \leq 10^{-4} \)
and \( \max(a_i^{\top} d_{k} - b_i, 0) \leq 10^{-4}(1 + \abs{b_i}) \) for all \( i = 1, \ldots, N_c \).

\bibliographystyle{ieeetr}
\bibliography{references}

\end{document}